\documentclass[a4paper, 11pt]{article}

\usepackage{authblk}

\usepackage{amsmath,amsthm,amssymb,enumerate}
\usepackage{cite}
\usepackage[margin=2cm]{geometry}
\usepackage[hidelinks]{hyperref}
\usepackage[capitalise]{cleveref}
\usepackage{tikz}
\usetikzlibrary{calc,intersections,through,backgrounds,arrows,patterns}
\usepackage{subcaption}

\tikzstyle{grid_point}=[circle,draw=black!50,text=white,inner sep=0.5mm,font=\scriptsize]
\tikzstyle{grid_white_point}=[circle,draw=black!50,fill=white,text=white,inner sep=0.5mm,font=\scriptsize]
\tikzstyle{true_point}=[circle,draw=black!60,fill=black!60,text=white,inner sep=0.5mm,font=\scriptsize]
\tikzstyle{false_point}=[circle,draw=black!70,text=white,inner sep=0.5mm,font=\scriptsize]
\tikzstyle{intersection_point}=[circle,draw=black!90,fill=black!90,text=white,inner sep=0.4mm,font=\scriptsize]

\tikzstyle{oriented_segment} = [draw,line width=1pt,arrows=-latex',black!80]
\tikzstyle{dashed_line} = [draw,dashed,line width=0.4pt,black!60]

\makeatletter
\tikzset{
        hatch distance/.store in=\hatchdistance,
        hatch distance=5pt,
        hatch thickness/.store in=\hatchthickness,
        hatch thickness=5pt
        }
\pgfdeclarepatternformonly[\hatchdistance,\hatchthickness]{north east hatch}
    {\pgfqpoint{-1pt}{-1pt}}
    {\pgfqpoint{\hatchdistance}{\hatchdistance}}
    {\pgfpoint{\hatchdistance-1pt}{\hatchdistance-1pt}}%
    {
        \pgfsetcolor{\tikz@pattern@color}
        \pgfsetlinewidth{\hatchthickness}
        \pgfpathmoveto{\pgfqpoint{0pt}{0pt}}
        \pgfpathlineto{\pgfqpoint{\hatchdistance}{\hatchdistance}}
        \pgfusepath{stroke}
    }
\pgfdeclarepatternformonly[\hatchdistance,\hatchthickness]{north west hatch}
    {\pgfqpoint{-1pt}{-1pt}}
    {\pgfqpoint{\hatchdistance}{\hatchdistance}}
    {\pgfpoint{\hatchdistance-1pt}{\hatchdistance-1pt}}%
    {
        \pgfsetcolor{\tikz@pattern@color}
        \pgfsetlinewidth{\hatchthickness}
        \pgfpathmoveto{\pgfqpoint{0pt}{\hatchdistance}}
        \pgfpathlineto{\pgfqpoint{\hatchdistance}{0pt}}
        \pgfusepath{stroke}
    }
\makeatother

\tikzstyle{small_text} = [font=\scriptsize]

\theoremstyle{plain}
\newtheorem{theorem}{Theorem}

\newtheorem{corollary}[theorem]{Corollary}
\newtheorem{example}[theorem]{Example}
\newtheorem{lemma}[theorem]{Lemma}

\newtheorem{claim}[theorem]{Claim}
\theoremstyle{definition}
\newtheorem{definition}{Definition}[section]

\newcommand{\Conv}{\textup{Conv}}
\newcommand{\Vertic}{\textup{Vert}}
\newcommand{\dir}{\overrightarrow}

\newcommand{\superb}{proper }

\newcommand{\andl}{\wedge}

\newcommand{\gridMN}{\mathcal{G}_{m,n}}

\newcommand{\grid}{\mathcal{G}}

\newcommand{\tr}[3]{{#1}{#2}{#3}}
\newcommand{\ortr}[3]{\overrightarrow{{#1}{#2}{#3}}}

\title{A characterization of 2-threshold functions via pairs of prime segments}
\author[1]{Elena Zamaraeva}
\author[2]{Jovi{\v{s}}a {\v{Z}}uni{\'c}}
\affil[1]{Mathematics Institute, University of Warwick, UK}
\affil[2]{Mathematical Institute, Serbian Academy of Sciences, Serbia}

\definecolor{light_grey}{rgb}{0.85,0.85,0.85}
\definecolor{light_light_grey}{rgb}{0.95,0.95,0.95}

\ifpdf
\hypersetup{
  pdftitle={On the number of 2-threshold functions},
  pdfauthor={E. Zamaraeva and J. {\v{Z}}uni{\'c}}
}
\fi

\begin{document}

\maketitle

\begin{abstract}
A $\{0,1\}$-valued function on a two-dimensional rectangular grid is called threshold if its sets of zeros and ones are separable by a straight line. 
In this paper we study 2-threshold functions, i.e. functions representable as the conjunction of two threshold functions.
We provide a characterization of 2-threshold functions by pairs of oriented prime segments, where each such segment is defined by an ordered pair of adjacent integer points.

\end{abstract}


\tableofcontents

\section{Introduction}
Denote a two-dimensional rectangular grid by $\gridMN = \{0,\dots,m-1\} \times \{0,\dots,n-1\}$.
%
A function $f$ mapping $\gridMN$ to $\{0,1\}$ is called \textit{threshold} if there exist natural numbers $a_0,a_1,a_2$ such that for each $(x_1,x_2) \in \gridMN$
$$
f(x_1,x_2) = 1 \iff a_1x_1 + a_2x_2 \geq a_0.
$$
The inequality $a_1x_1 + a_2x_2 \geq a_0$ is called a \textit{threshold inequality} for the function $f$.
We also say that the set of  true points $M_1(f)$ and the set of false points $M_0(f)$ are separable by the line $a_1x_1 + a_2x_2 = a_0$.

It is easy to see that $f$ is threshold if and only if
$$
\Conv(M_0(f)) \cap \Conv(M_1(f)) = \emptyset,
$$
where $\Conv(T)$ denotes the convex hull of a given set of points $T$.

For a natural number $k \geq 2$, a function $f :\gridMN \rightarrow \{0,1\}$ is called \textit{$k$-threshold} if there exist at most $k$ threshold functions $f_1, \dots, f_k$ such that $f$ coincides with the conjunction of the functions $f_1,\dots,f_k$, i.e. $f= f_1 \land \dots \land f_k$.
We also say that the functions $f_1,\dots,f_k$ \emph{define} the $k$-threshold function $f$.
A $k$-threshold function is called \emph{proper $k$-threshold} if it is not $(k-1)$-threshold.

Threshold functions refer to the linear partitions of a given set of points.
One also studies non-linear partitions by circles \cite{huxley-focm,huxley-lms}, convex curves \cite{ivic}, arbitrary curves \cite{Zunic1998} in 2-dimension and spheres \cite{zunic-aam} and surfaces \cite{zunic-ieee} in higher dimensions.
In particular, polynomial threshold functions are considered in \cite{anthony, bruck, boolean, boolean-2}.
It is worth to note, that $k$-threshold functions represent the partition of the domain by at most $k$ straight lines (halfspaces) \textit{in general position}, and hence have richer structure than many other studied partitions by multiple lines or surfaces such as parallel hyperplanes \cite{Zunic2007} or $d$-dimensional spheres centered at the same point.

In machine learning theory learning of Boolean $k$-threshold functions was studied, for instance, in \cite{Baum1991, Hegedus1997, Kwek1998, Klivans2004}.
Lower bounds on the complexity of learning threshold, $k$-threshold functions, and some related geometric objects were derived in \cite{Maass1994}.
An efficient algorithm of learning with membership queries for $k$-threshold functions on the two-dimensional grid was developed in \cite{Maass}.
Structural properties of threshold and $k$-threshold functions affecting their learning complexity were also studied in \cite{Zolotykh,Shevchenko3,Zamaraeva2016, Zamaraeva2017,elena-3,elena-4}.

In the realm of exact learning, \emph{specifying} (\emph{teaching}) \emph{sets} and \emph{essential points} play a special role.
Let $C$ be a class of functions mapping $S$ to $\{0,1\}$, and let $f$ be a function from $C$.
A set of points $T \subseteq S$ is called a \emph{teaching} or \emph{specifying set for $f$ with respect to $C$} if no other function from $C$ coincides with $f$ in all points of $T$.
The number of points in a minimum specifying set of $f$ (with respect to $C$) is called the \emph{specification number of $f$} (\emph{with respect to $C$}).
Clearly, the target function $f$ cannot be identified without learning the values of all points in a specifying set for $f$.

A point $x \in  S$ is called \emph{essential for $f$ with respect to $C$} if there exists a function $g \in C$ such that $g(x) \neq f(x)$ and $g$ coincides with $f$ on $S \setminus \{x\}$.
It is easy to see that the set of essential points of $f$ is a subset of any specifying set of $f$.
Furthermore, it is known that the set of essential points of a threshold function is a specifying set by itself (see e.g. \cite{AnthonyBrightwellShaweTaylor1995, Shevchenko}).
The specification number and essential points of Boolean threshold functions were studied in \cite{AnthonyBrightwellShaweTaylor1995}, and the non-Boolean case was considered in \cite{Zolotykh, Shevchenko, Shevchenko2, Shevchenko3}.
The specifying sets of $k$-threshold functions were studied in \cite{Zamaraeva2016,Zamaraeva2017}.

In digital geometry, the problem of polyhedral separability can be formulated in terms of $k$-threshold functions as follows: given a domain $S$, a finite set of points $T \subseteq S$, and a positive integer $k$,
does there exist a $k$-threshold function $f$ on $S$ such that $T$ is the set of true points of $f$?
The problem of polyhedral separability is widely investigated (see \cite{Edelsbrunner1988, Megiddo1988, Bennett1993, Astorino2002, Dundar2008, Astorino2013, Gerard2017, Gerard20172}).
In particular, in \cite{Bennett1993} the authors studied bilinear separation which is closely related to $2$-threshold functions, and the papers \cite{Gerard2017, Gerard20172, Gerard2020} are devoted to the polyhedral separability problem in two- and three-dimensional spaces.

Threshold functions admit various representations and usually the choice of specific description depends on the restrictions of a particular application. 
The most natural way of defining threshold functions is via threshold inequalities. 
However, for a given threshold function there are continuously many threshold inequalities, and given two linear inequalities it is not obvious whether they define the same threshold function or not. 

Another way of describing threshold functions is via essential points. 
The set of essential points of a threshold function $f$ together with the values of $f$ in all these points uniquely identifies $f$ in the class of threshold functions. 
However, for 2-threshold functions this approach does not work.
In contrast to threshold functions, the set of essential points of a 2-threshold function does not always specify it (see, for instance, \cite{Zamaraeva2016}).

A useful characterization of two-dimensional threshold functions via oriented prime segments~(i.e., ordered pairs of adjacent integer points) was provided in \cite{Koplowitz}.
In that and the subsequent works \cite{Acketa, Haukkanen} the relation between threshold functions and prime segments was utilized to estimate the number of threshold functions asymptotically.
It is important to note that the endpoints of the segment defining a threshold function $f$, are essential for $f$, and hence, $f$ can be defined by an \emph{ordered} pair of \emph{adjacent} essential points.
Such a representation is space-optimal and requires less memory storage than the representation via the set of all essential points, as the latter consists of 3 or 4 points \cite{Shevchenko2}.

Since a $2$-threshold function is the conjunction of two threshold functions, it is also possible to define it via a pair of threshold functions or the corresponding pair of prime segments.
A drawback of such representation is that the same function, in general, can be defined by many different pairs of threshold functions, and therefore by many different pairs of prime segments.
Moreover, the points which are essential for a threshold function, can be inessential for its conjunction with another threshold function.
In this paper we overcome these challenges by introducing pairs of oriented prime segments with certain properties which we call \emph{proper} pairs of segments.
We show that the endpoints of the segments from a proper pair of segments defining a $2$-threshold function $f$ are essential for $f$.
Furthermore, we establish a bijection between the proper $2$-threshold functions which have a true point on the boundary of $\gridMN$ and the proper pairs of segments defining these functions.
In the subsequent work \cite{part2} this bijection is used for the asymptotic enumeration of $2$-threshold functions.

Finally, if we interpret a $2$-threshold function $f$ on $\gridMN$ as a convex integer polygon $\Conv(M_1(f))$ in $\gridMN$, the proposed representation of $f$ by a proper pair of segments provides a $O(\log(m+n))$ memory space representation, while the general representation scheme based on integer polygons \cite{Zunic1995} in $\gridMN$ requires $O\left((m+n)^{\frac{2}{3}} \log(m+n)\right)$ memory space.

The organization of the paper is as follows.
All preliminary information can be found in \cref{sec:preliminaries}.
In \cref{sec:oriented_segments} we describe and adapt to our purposes the bijection between oriented prime segments and non-constant threshold functions from \cite{Koplowitz}.
In \cref{sec:pair_oriented_segments} we introduce proper pairs of segments and show that any proper $2$-threshold function can be defined by a proper pair of segments.
In \cref{sec:proper_functions} we prove that for a proper $2$-threshold function with a true point on the boundary of $\gridMN$ there exists a unique proper pair of segments that defines the function.

\section{Preliminaries}
\label{sec:preliminaries}

In this paper we denote points on the plane by capital letters $A,B,C$, etc. 
For two sets of points $S_1$, $S_2$ we denote by $d(S_1, S_2)$ the (Euclidean) distance between the sets,
that is, the minimum distance between two points $A \in S_1$ and $B \in S_2$.
When a set consists of a single point we omit $\{ \}$ and write simply $d(A, S_2)$ or $d(A,B)$
to denote the distance between the point $A$ and set $S_2$ or the distance between the points $A$ and $B$, respectively.
For two distinct points $A$, $B$ we denote by $\ell(AB)$ the line which passes through these points.

A point $A = (x,y)$  is \emph{integer}, if both of its coordinates $x$ and $y$ are integer.
Two points $A$, $B$ are called \textit{adjacent} if they are integer and there is no other integer points on $AB$.
A segment with adjacent endpoints is called \textit{prime}.

We say that the points $A_1, A_2, \dots, A_n$ are in convex position if $\{A_1,\dots,A_n\} = \Vertic(\Conv(\{A_1,\dots,A_n\}))$.
We also denote by $P(f)$ the convex hull of $M_1(f)$, that is $P(f) = \textup{Conv}(M_1(f))$.

\subsection{Segments, triangles, quadrilaterals and their orientation}
\label{sec:seg_tr_orientation}

We often denote a \emph{convex} polygon 
by a sequence of its vertices in either clockwise or counterclockwise order. 
For example, by $AB$, $ABC$, and $ABCD$ we denote, respectively, the segment with endpoints $A,B$,
the triangle with vertices $A,B,C$, and the convex quadrilateral with vertices $A,B,C,D$ and edges $AB$, $BC$, $CD$, $DA$.
When the order of vertices is important, we call the polygon or segment \emph{oriented} and add an arrow in the notation, that is, 
$\dir{AB}$, $\dir{ABC}$, $\dir{ABCD}$ denote the oriented segment, the oriented triangle, and the oriented convex quadrilateral, respectively.

Let $A=(a_1,a_2),B=(b_1,b_2),C=(c_1,c_2)$ be distinct points on the plane. It is a basic fact that
$A,B,C$ are collinear if and only if $\Delta = 0$, where 
$$
\Delta =\begin{vmatrix}
a_1 & a_2 & 1 \\ 
b_1 & b_2 & 1 \\ 
c_1 & c_2 & 1
\notag
\end{vmatrix}.
$$
The oriented triangle $\ortr{A}{B}{C}$ is called \textit{clockwise} if $\Delta < 0$ and \textit{counterclockwise} if $\Delta > 0$.
Geometrically, an oriented triangle $\ortr{A}{B}{C}$ is clockwise (resp. counterclockwise) if its vertices $A, B, C$, in order, rotate clockwise (resp. counterclockwise) around the triangle's center.
Some properties of oriented triangles easily follow from the definition:

\begin{claim}\label{prop:theSameOrient}
Let $\ell$ be a line and let $A,B$ be two distinct points on $\ell$.
Then for any two points $C,D \notin \ell$ the orientations of the triangles $\ortr{A}{B}{C}$ and $\ortr{A}{B}{D}$ are the same if and only if $\ell \cap CD = \emptyset$ (see Fig. \ref{fig:co-oriented-triangles} and \ref{fig:opposite-triangles}).
\end{claim}

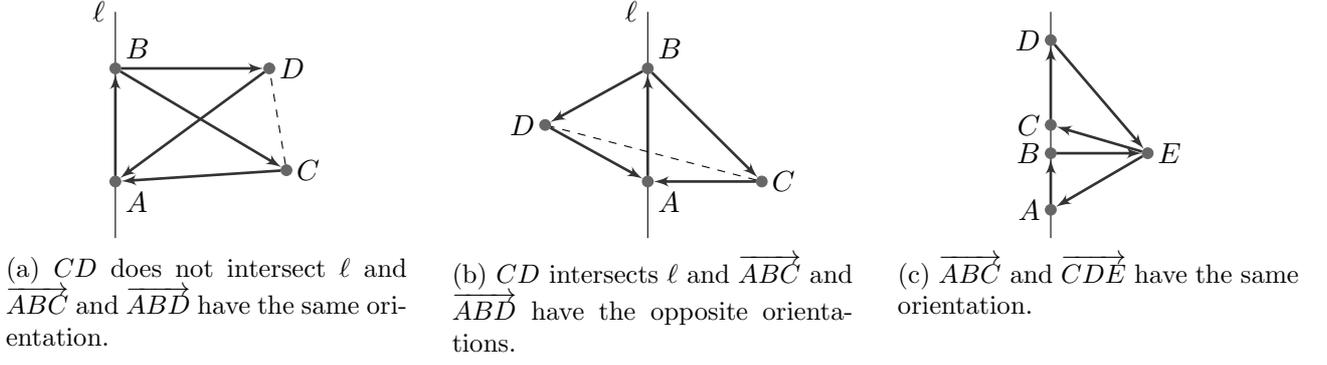
\begin{figure}
\begin{subfigure}[t]{0.31\textwidth}
\centering
\begin{tikzpicture}[scale=0.75]
	
	\draw[black] (0,0) -- (0,4);	
	\node[true_point] (A) at (0,1) {};
	\node[true_point] (B) at (0,3) {};
	\node[true_point] (C) at (3,1.2) {};
	\node[true_point] (D) at (2.7,3) {};
	\draw[oriented_segment] (A) -- (B);
	\draw[oriented_segment] (B) -- (C);
	\draw[oriented_segment] (C) -- (A);
	\draw[oriented_segment] (B) -- (D);
	\draw[oriented_segment] (D) -- (A);
	\draw (A) node [below right] {$A$};
	\draw (B) node [above right] {$B$};
	\draw (C) node [right] {$C$};
	\draw (D) node [right] {$D$};
	\draw (0,4) node [left] {$\ell$};
	\draw[dashed] (C) -- (D);
	
\end{tikzpicture}
\caption{$CD$ does not intersect $\ell$ and $\protect\ortr{A}{B}{C}$ and $\protect\ortr{A}{B}{D}$ have the same orientation.}
\label{fig:co-oriented-triangles}
\end{subfigure}\hfill\hfill
\begin{subfigure}[t]{0.31\textwidth}
\centering
\begin{tikzpicture}[scale=0.75]
	\draw[black] (0,0) -- (0,4);	
	\node[true_point] (A) at (0,1) {};
	\node[true_point] (B) at (0,3) {};
	\node[true_point] (C) at (2,1) {};
	\node[true_point] (D) at (-1.8,2) {};
	\draw[oriented_segment] (A) -- (B);
	\draw[oriented_segment] (B) -- (C);
	\draw[oriented_segment] (C) -- (A);
	\draw[oriented_segment] (B) -- (D);
	\draw[oriented_segment] (D) -- (A);
	\draw (A) node [below right] {$A$};
	\draw (B) node [above right] {$B$};
	\draw (C) node [right] {$C$};
	\draw (D) node [left] {$D$};
	\draw (0,4) node [left] {$\ell$};
	\draw[dashed] (C) -- (D);
\end{tikzpicture}
\caption{$CD$ intersects $\ell$ and $\protect\ortr{A}{B}{C}$ and $\protect\ortr{A}{B}{D}$ have the opposite orientations.}
\label{fig:opposite-triangles}
\end{subfigure}\hfill\hfill
\begin{subfigure}[t]{0.31\textwidth}
\centering
\begin{tikzpicture}[scale=0.75]
	\draw[black] (0,0) -- (0,4);	
	\node[true_point] (A) at (0,0.5) {};
	\node[true_point] (B) at (0,1.5) {};
	\node[true_point] (C) at (0,2) {};
	\node[true_point] (D) at (0,3.5) {};
	\node[true_point] (E) at (1.7,1.5) {};
	\draw[oriented_segment] (A) -- (B);
	\draw[oriented_segment] (B) -- (E);
	\draw[oriented_segment] (E) -- (A);
	\draw[oriented_segment] (C) -- (D);
	\draw[oriented_segment] (D) -- (E);
	\draw[oriented_segment] (E) -- (C);
	\draw (A) node [left] {$A$};
	\draw (B) node [left] {$B$};
	\draw (C) node [left] {$C$};
	\draw (D) node [left] {$D$};
	\draw (E) node [right] {$E$};
\end{tikzpicture}
\caption{$\protect\ortr{A}{B}{C}$ and $\protect\ortr{C}{D}{E}$ have the same orientation.}
\label{fig:co-oriented-triangles-E}
\end{subfigure}
\caption{The orientation of the triangles depending on the positions of points}
\end{figure}

\begin{claim}
\label{prop:collinear_segments_and_point}
Let $\dir{AB},\dir{CD}$ be two collinear segments with the same orientation.
Then for any point $E \notin \ell(AB)$ the triangles $\ortr{A}{B}{E}$ and $\ortr{C}{D}{E}$ have the same orientation (see \cref{fig:co-oriented-triangles-E}).
\end{claim}

\begin{claim}
\label{prop:clockwise_triangles}
Let $A$, $B$, $C$, $D$ be four distinct points such that $\ortr{A}{B}{D}$, $\ortr{B}{C}{D}$, $\ortr{C}{A}{D}$ are clockwise (resp. counterclockwise) triangles.
Then $\ortr{A}{B}{C}$ is a clockwise (resp. counterclockwise) triangle.
\end{claim}
\begin{proof}
We will prove the statement for clockwise triangles, the counterclockwise case is symmetric.
Denote $\mathcal{P} = \Conv(\{A,B,C,D\})$.
First, we show that $D$ is not a vertex of $\mathcal{P}$.
Suppose, to the contrary, that $D$ is a vertex of $\mathcal{P}$, then two of the segments $CD$, $BD$, $AD$ are edges of $\mathcal{P}$.
The triangle $\ortr{C}{A}{D}$ is clockwise, hence the triangle $\ortr{C}{D}{A}$ is counterclockwise and the points $A$ and $B$ 
are separated by $\ell(CD)$,
and therefore $CD$ is not an edge of $\mathcal{P}$.
Similarly, the opposite orientations of the triangles $\ortr{A}{B}{D}$ and $\ortr{B}{D}{C}$ imply that $BD$ is not an edge of $\mathcal{P}$.
The above contradicts the assumption that two of the segments $CD$, $BD$, $AD$ are edges of $\mathcal{P}$, and therefore
$D$ is not a vertex of $\mathcal{P}$ and $\mathcal{P}$ is the triangle with vertices $A,B,C$.
Finally, since $D$ is an interior point of $\mathcal{P}$, the points $C$ and $D$ lie on the same side from $\ell(AB)$, hence the triangles $\ortr{A}{B}{D}$ and $\ortr{A}{B}{C}$ have the same orientation, i.e. $\ortr{A}{B}{C}$ is clockwise, as required (see \cref{fig:D-inside-triangle}).
\end{proof}

\begin{figure}
\centering
\begin{tikzpicture}[scale=1]
	
	\node[true_point] (A) at (-1,0) {};
	\node[true_point] (B) at (1,2) {};
	\node[true_point] (C) at (3,0) {};
	\node[true_point] (D) at (1,1) {};
	\draw[oriented_segment] (A) -- (B);
	\draw[oriented_segment] (B) -- (C);
	\draw[oriented_segment] (C) -- (A);
	\draw[black] (B) -- (D);
	\draw[black] (D) -- (A);
	\draw[black] (D) -- (C);
	\draw (A) node [left] {$A$};
	\draw (B) node [above] {$B$};
	\draw (C) node [right] {$C$};
	\draw (D) node [below] {$D$};
	
\end{tikzpicture}
\caption{$\protect\ortr{A}{B}{C}$ has the same orientation as $\protect\ortr{A}{B}{D}$, $\protect\ortr{B}{C}{D}$, and $\protect\ortr{C}{A}{D}$.}
\label{fig:D-inside-triangle}
\end{figure}

It is clear, that for a given convex oriented quadrilateral $\dir{ABCD}$ the orientation of the triangles $\ortr{A}{B}{C}$, $\ortr{B}{C}{D}$, $\ortr{C}{D}{A}$, and $\ortr{D}{A}{B}$ is the same and determines the orientation of $\dir{ABCD}$.
Moreover, the opposite is also true.

\begin{claim}
\label{prop:convex_rectangle}
Let $\ortr{A}{B}{C}$, $\ortr{B}{C}{D}$, $\ortr{C}{D}{A}$, $\ortr{D}{A}{B}$ be clockwise (resp. counterclockwise) triangles.
Then $\Conv(\{A,B,C,D\})$ is a quadrilateral with edges $AB$, $BC$, $CD$, and $DA$ and the orientation of $\dir{ABCD}$
is clockwise (resp. counterclockwise).
\end{claim}
\begin{proof}
Clearly, $A,B,C,$ and $D$ are pairwise distinct points.
Let \linebreak $\mathcal{P} = \Conv(\{A,B,C,D\})$.
Since $\ortr{A}{B}{C}$ and $\ortr{D}{A}{B}$ are triangles with the same orientation, we conclude that $C$ and $D$ lie on the same side of $\ell(AB)$,
and therefore $\ell(AB)$ is a tangent to $\mathcal{P}$ and $AB$ is an edge of $\mathcal{P}$.
By similar arguments each of the segments $BC$, $CD$, and $DA$ is an edge of $\mathcal{P}$, hence $\mathcal{P}$ is a quadrilateral.
Finally, the orientation of the triangles implies that $\dir{ABCD}$ has the same orientation as the orientation of the triangles.
\end{proof}

\subsection{Convex sets and their tangents}
\label{sec:tangents}

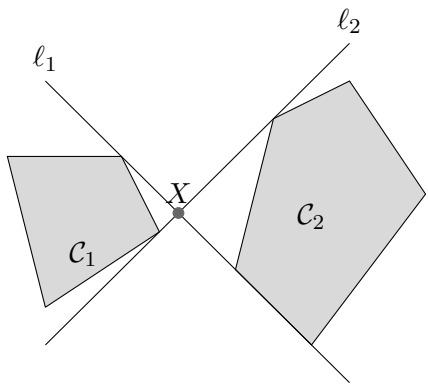
\begin{figure}
\centering
\begin{tikzpicture}[scale=0.5]
  
   	\draw[black,fill=light_grey] (1,2) -- (4,4) -- (3,6) -- (0,6) -- (1,2);
    	\draw[black,fill=light_grey] (6,3) -- (8,1) -- (11,5) -- (9,8) -- (7,7) -- (6,3);
 	\draw (1,1) -- (9,9);
 	\draw (1,8) -- (9,0);
	\node[true_point] (intersectionPoint) at (4.5,4.5) {};
 	\draw (4.5,4.5) node [above] {$X$};
	\draw (8,5) node [below] {$\mathcal{C}_2$};
	\draw (2,4) node [below] {$\mathcal{C}_1$};
	\draw (9,9) node [above] {$\ell_2$};
	\draw (1,8) node [above] {$\ell_1$};

\end{tikzpicture}
\caption{$\ell_1$ is the right tangent from $X$ to $\protect\mathcal{C}_1$ and to $\protect\mathcal{C}_2$, and the right inner common tangent for $\protect\mathcal{C}_1$ and $\protect\mathcal{C}_2$. 
$\ell_2$ is the left tangent from $X$ to $\protect\mathcal{C}_1$ and to $\protect\mathcal{C}_2$, and the left inner common tangent for $\protect\mathcal{C}_1$ and $\protect\mathcal{C}_2$.}
\label{fig:inner_common_tangent}
\end{figure}

Let $\mathcal C$ be a convex set. 
A convex polygon $\mathcal P$ is called \textit{circumscribed} about $\mathcal C$ if for every edge $AB$ of $\mathcal P$ the line $\ell(AB)$ is a tangent to $\mathcal C$ and $AB \cap \mathcal{C} \neq \emptyset$.

Let $\mathcal C_1$ and $\mathcal C_2$ be two disjoint convex sets.
A line $\ell$ is called an \textit{inner common tangent} to $\mathcal C_1$ and $\mathcal C_2$ if it is a tangent to both of them, 
and $\mathcal C_1$ and $\mathcal C_2$ are separated by $\ell$.

Let $\ell$ be a tangent to a convex set $\mathcal C$, and let $X$ be a point in $\ell \setminus \mathcal{C}$. 
Then $\ell$ is called a \textit{right} (resp. \textit{left}) \textit{tangent} from $X$ to $\mathcal{C}$ if for any points 
$Y \in \mathcal{C} \cap \ell$ and $Z \in \mathcal{C} \setminus \ell$ the triangle $\ortr{X}{Y}{Z}$ is counterclockwise (resp. clockwise).
The following claim is a simple consequence of the above definition.

\begin{claim}
\label{prop:tangent_on_line}
Let $\ell$ be the right (resp. left) tangent from a point $X$ to a convex set $\mathcal C$, and let $Y \in \ell$.
Then $\ell$ is the right (resp. left) tangent from $Y$ to $\mathcal C$ if and only if $XY \cap \mathcal{C} = \emptyset$.
\end{claim}

Let $\ell$ be an inner common tangent to two disjoint convex sets $\mathcal C_1$ and $\mathcal C_2$, and let 
$A, B$ be two points such that $A \in \mathcal{C}_1 \cap \ell$ and $B \in \mathcal{C}_2 \cap \ell$.
Then $\ell$ is called the \textit{right} (resp. \textit{left}) \textit{inner common tangent} to $\mathcal{C}_1$ and $\mathcal{C}_2$ if $\ell$ 
is the right (resp. left) tangent from $A$ to $\mathcal{C}_2$, and the right (resp. left) tangent from $B$ to $\mathcal{C}_1$
(see \cref{fig:inner_common_tangent}).
It is easy to see that any pair of disjoint convex sets has exactly 
one right and exactly one left inner common tangent.

\section{Oriented prime segments and threshold functions}
\label{sec:oriented_segments}

\begin{definition}\label{def:segment_defines_f}
Let $A$ and $B$ be two adjacent points in $\gridMN$.
We say that $\dir{AB}$ \textit{defines} a function $f: \gridMN \rightarrow \{0,1\}$ if:
\begin{enumerate}
\item $f(A) = 1, f(B) = 0$;
\item for any $X \in \gridMN \cap \ell(AB)$ we have $f(X) = 1$ if and only if $d(A,X) < d(B,X)$;
\item for any $X \in \gridMN \setminus \ell(AB)$ we have $f(X) = 1$ if and only if $\ortr{A}{B}{X}$ is a counterclockwise triangle.
\end{enumerate}
The function defined by $\dir{AB}$ will be denoted as $f_{\dir{AB}}$.
\end{definition}

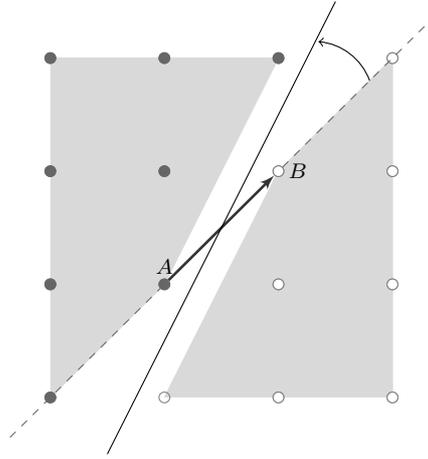
\begin{figure}
\centering
\begin{tikzpicture}[scale=1.5]

	\draw[light_grey, fill=light_grey] (0,0) -- (1,1) -- (2,3) -- (0,3);
	\draw[light_grey, fill=light_grey] (1,0) -- (2,2) -- (3,3) -- (3,0);

	\foreach \x in {0,...,3}
    	\foreach \y in {0,...,3}
    	{
    		\node[grid_point] (\x\y) at (\x,\y) {};
    	}
   
     	\foreach \y in {0,...,3}
    	{
    		\node[grid_white_point] (3\y) at (3,\y) {};
    	}
    
    	\foreach \y in {0,...,2}
    	{
    		\node[grid_white_point] (2\y) at (2,\y) {};
   	 }
    
    	\node[grid_white_point] (01) at (0,1) {};

   	 \foreach \y in {0,...,3}
   	 {
    		\node[true_point] (0\y) at (0,\y) {};
    	}
    
   	 \foreach \y in {1,...,3}
    	{
    		\node[true_point] (1\y) at (1,\y) {};
    	}
    
    \node[true_point] (23) at (2,3) {};
    	
    \node[small_text, above] at (11) {$A$};
    \node[small_text, right] at (22) {$B$};
    \draw[oriented_segment, name path=AB] (11) -- (22);
    \path (11) -- (22) coordinate[pos=-1.5](dd) coordinate[pos=2.5](ff);
	
	\draw[dashed_line, name path=lineAB] (dd) -- (11);
	\draw[dashed_line] (22) -- (ff);
	\draw[black] (0.5,-0.5) -- (2.5,3.5);	
	\draw [->, black] (2.8,2.8) arc (20:85:15pt); 
	
\end{tikzpicture}
\caption{$\protect\dir{AB}$ defines the threshold function $f$ where $\Conv(M_1(f))$ and $\Conv(M_0(f))$ are the left and right grey regions respectively. 
The black circles are the true points of $f$, the white circles are the false points of $f$.
The dashed line is the left inner common tangent to $\Conv(M_0(f))$ and $\Conv(M_1(f))$. 
The solid line is a separating line for $f$.}
\label{fig:dir_seg_definition}
\end{figure}

The following statement is an immediate consequence of \cref{def:segment_defines_f}.

\begin{claim}
\label{prop:points_on_line}
Let $\dir{AB}$ be a prime segment in $\gridMN$ and let $f = f_{\dir{AB}}$ be the function on $\gridMN$ defined by $\dir{AB}$.
Then for any $C \in \ell(AB) \cap \gridMN$ we have either $f(C) = 1$ and $A \in BC$ or $f(C) = 0$ and $B \in AC$.
\end{claim}

In \cite{Koplowitz} authors, in different terms, showed that a function $f_{\dir{AB}}$ defined by an oriented prime segment $\dir{AB}$ is threshold and the line $\ell(AB)$ is an inner common tangent to the convex hulls of the sets of true and false points of $f$.
For the convenience, the following theorem partly repeats the result from \cite{Koplowitz}, thus adapting it to our purposes and making our exposition self-contained.

\begin{theorem}
\label{th:dir_segment_props}
Let $A$ and $B$ be two adjacent points in $\gridMN$ and let $f=f_{\dir{AB}}$.
Then 
\begin{enumerate}
	\item[(1)] $f$ is a threshold function;
	\item[(2)] $A$ and $B$ are essential points of $f$;
	\item[(3)] $\ell(AB)$ is the left inner common tangent to $\Conv(M_1(f))$ and $\Conv(M_0(f))$.
\end{enumerate}
\end{theorem}
\begin{proof}
First we prove (1).
Indeed, if we consider the line $\ell(AB)$ and turn it counterclockwise slightly around the middle of the segment $AB$ to not intersect any integer points then we obtain a separating line for $f$, hence $f$ is a threshold function (see  \cref{fig:dir_seg_definition}).

Let us now prove (2).
Consider the line $\ell(AB)$ and turn it counterclockwise slightly around the point $A$ to not intersect any integer points except $A$.
The obtained line separates $M_1(f) \setminus \{A\}$ and $M_0(f) \cup \{A\}$, 
and witnesses that the function that differs from $f$ in the unique point $A$ is threshold.
Therefore, the point $A$ is essential for $f$.
Similarly, one can show that $B$ is also essential for $f$.

Now we prove (3).
First, it is easy to see that $\ell(AB)$ is a tangent to both $\Conv(M_1(f))$ and $\Conv(M_0(f))$.
Furthermore, since $\Conv(M_1(f))$ and $\Conv(M_0(f))$ are separated by $\ell(AB)$, we conclude that $\ell(AB)$ is an inner common tangent for $\Conv(M_1(f))$ and $\Conv(M_0(f))$.
Now, by \cref{def:segment_defines_f}, for any $X \in M_1(f) \setminus \ell(AB)$ the triangle $\ortr{B}{A}{X}$ is clockwise, and for any $X \in M_2(f) \setminus \ell(AB)$ the triangle $\ortr{A}{B}{X}$ is clockwise.
Hence, $\ell(AB)$ is a left tangent from $B$ to $\Conv(M_1(f))$ and from $A$ to $\Conv(M_2(f))$, i.e. $\ell(AB)$ is the left inner common tangent for $\Conv(M_1(f))$ and $\Conv(M_0(f))$.
\end{proof}

In \cite{Koplowitz} authors also proved a bijection between oriented prime segments and non-constant threshold functions:

\begin{theorem}[\cite{Koplowitz}]
There is one-to-one correspondence between oriented prime segments in $\gridMN$ and non-constant threshold functions on $\gridMN$.
\end{theorem}

\begin{corollary}
\label{lem:dir_seg_exist}
Let $f$ be a non-constant threshold function on $\gridMN$.
Then there exists a unique prime segment $AB$ with $A,B \in \gridMN$ such that $f = f_{\dir{AB}}$.
\end{corollary}

\section{Proper pairs of oriented prime segments}
\label{sec:pair_oriented_segments}

The defining threshold functions via oriented prime segments can be naturally extended to $2$-threshold functions.

\begin{definition}
We say that a pair of oriented prime segments $\dir{AB}, \dir{CD}$ in $\gridMN$ \textit{defines} a $2$-threshold function $f$ on $\gridMN$ if
$$
f = f_{\dir{AB}} \land f_{\dir{CD}}.
$$
\end{definition}

However, the proposed representation of a $2$-threshold function has two drawbacks.
First, although the points $A,B$ and $C,D$ are essential for $f_{\dir{AB}}$ and $f_{\dir{CD}}$ respectively, they are not necessarily essential for $f= f_{\dir{AB}} \land f_{\dir{CD}}$.
The following example demonstrates this fact.

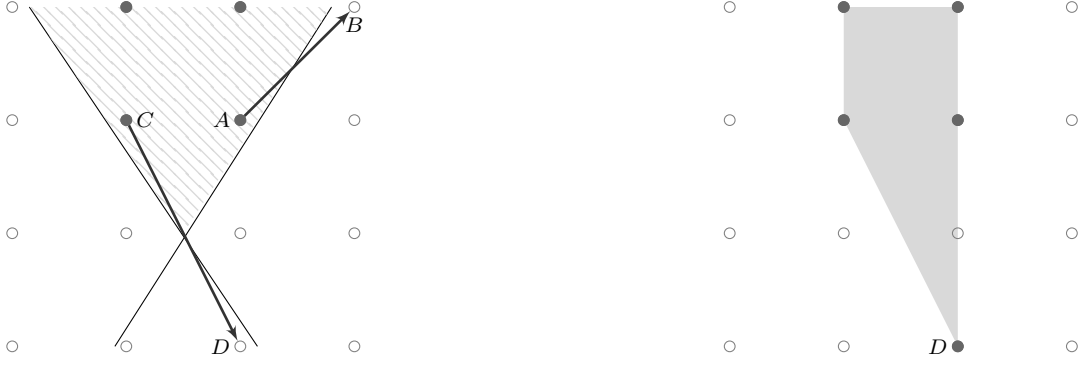
\begin{figure}
\begin{subfigure}[t]{0.45\textwidth}
\captionsetup{aboveskip=20pt}
\centering
\begin{tikzpicture}[scale=1.5]

   	
   	\draw (0.9,0) -- (2.8,3);
   	\draw (0.15,3) -- (2.15, 0);
   	
   	\pattern[pattern=north west hatch, hatch distance=2mm, pattern color = light_grey, hatch thickness=.5pt] (1.5,1) -- (2.8,3) -- (0.15,3);   

	\foreach \x in {0,...,3}
    	\foreach \y in {0,...,3}
    	{
    		\node[grid_point] (\x\y) at (\x,\y) {};
    	}
    	
    \node[true_point] at (1,2) {};
    \node[true_point] at (1,3) {};
    \node[true_point] at (2,2) {};
   	\node[true_point] at (2,3) {};
   	

   	\node[small_text, left] (A) at (2,2) {$A$};	
   	\node[small_text, below] (B) at (3,3) {$B$};	
    \node[small_text, right] (C) at (1,2) {$C$};	
   	\node[small_text, left] (D) at (2,0) {$D$};	
    \draw[oriented_segment, name path=AB] (22) -- (33);
    \draw[oriented_segment, name path=CD] (12) -- (20);

\end{tikzpicture}
\caption{The black and white points are the true and false points of $f$ respectively.
The pair of segments $\{\dir{AB},\dir{CD}\}$ defines $f$, i.e. $f = f_{\protect\dir{AB}} \land f_{\protect\dir{CD}}$.
The solid lines are separating lines of $f_{\protect\dir{AB}}$ and $f_{\protect\dir{CD}}$.
All integer points inside the stripped region and only them are the true points of $f$.}
\label{fig:non-essential-f}
\end{subfigure}\hfill
\begin{subfigure}[t]{0.45\textwidth}
\captionsetup{aboveskip=20pt}
\centering
\begin{tikzpicture}[scale=1.5]


   	\pattern[light_grey] (2,0) -- (1,2) -- (1,3) -- (2,3);   

	\foreach \x in {0,...,3}
    	\foreach \y in {0,...,3}
    	{
    		\node[grid_point] (\x\y) at (\x,\y) {};
    	}
    	
    \node[true_point] at (1,2) {};
    \node[true_point] at (1,3) {};
    \node[true_point] at (2,2) {};
   	\node[true_point] at (2,3) {};
   	\node[true_point] at (2,0) {};
    	
    \node[small_text, left] (D) at (2,0) {$D$};	
    	
\end{tikzpicture}
\caption{The black and white points are the true and false points of $g$ respectively.
The grey region is $\Conv(M_1(g))$.}
\label{fig:non-essential-g}
\end{subfigure}
\caption{A $2$-threshold function $f$ and the function $g$ that only differs from $f$ in the point $D$.}
\label{fig:non-essential}
\end{figure}

\begin{example}
Consider a $2$-threshold function $f$ on $\grid_{3,3}$ defined by a pair of segments $\{\dir{AB},\dir{CD}\}$, where $A=(2,2),B=(3,3),C=(1,2),D=(2,0)$ (see \cref{fig:non-essential-f}), and the function $g: \grid_{3,3} \rightarrow \{0,1\}$ such that $f$ and $g$ only differ in the point $D$.
The function $g$ is not $k$-threshold for any $k$ as the convex hull of its true points contains the false point $(2,1)$ (see \cref{fig:non-essential-g}).
Hence, the point $D$ is not essential for $f$.
\end{example}

\noindent
Second, the same $2$-threshold function can be expressed as the conjunction of different pairs of threshold functions, therefore its representation by pairs of oriented prime segments is not unique.

However, we may impose some restrictions on the pairs of oriented prime segments to exclude redundant pairs of segments defining the same function and guarantee that the endpoints of the segments are essential for the given function.

\begin{definition}
We say that a pair of oriented segments $\dir{AB}, \dir{CD}$ is \emph{proper} if the segments are prime and 
$$
f_{\dir{CD}}(A) = f_{\dir{CD}}(B) = f_{\dir{AB}}(C) = f_{\dir{AB}}(D) = 1.
$$
\end{definition}

The rest of the section is devoted to the analysis and properties of \superb pairs of segments.
We start by showing that the endpoints of the segments in a \superb pair of segments are essential for the function defined by this pair.

\begin{theorem}
\label{th:proper-is-essential}
Let $\{\dir{AB},\dir{CD}\}$ be a \superb pair of segments defining a $2$-threshold function $f$.
Then $A,B,C,D$ are essential points of $f$.
\end{theorem}
\begin{proof}
We will prove that $A$ is an essential point of $f$, for the points $B$, $C$, and $D$ the proof is similar.
Let $f'$ be the function which differs from $f_{\dir{AB}}$ in the unique point $A$, i.e. $f'(A) = 0$ and $f'(X) = f_{\dir{AB}}(X)$ for all $X \neq A$.
By \cref{th:dir_segment_props}, the point $A$ is essential for $f$, and hence $f'$ is a threshold function.
Then $f'' = f' \land f_{\dir{CD}}$ is a $2$-threshold function that coincides with $f$ in all points except $A$, as $f''(A) = 0$ and $f(A) = 1$, therefore $A$ is an essential point for $f$.
\end{proof}

The above theorem shows that a $2$-threshold function $f$ can be characterized by a partially ordered set of 4 essential points, while the representation of $f$ by its specifying set may require $9$ or even $\Theta(mn)$ points (see \cite{Zamaraeva2017}).
The proposed representation is also more space efficient than the representation of~$f$ by $\Conv(M_1(f))$ as a convex polygon, the former requires $O(\log(m+n))$ memory space while the latter needs $O\left((m+n)^{\frac{2}{3}} \log(m+n)\right)$ memory space \cite{Zunic1995}.

\begin{claim}
\label{prop:segment_end_points_enqualities}
Let $\{\dir{AB},\dir{CD}\}$ be a \superb pair of segments.
Then $A\neq D, C \neq B$, and $B \neq D$.
\end{claim}
\begin{proof}
The statement follows from the inequalities $\begin{aligned}f_{\dir{CD}}(A) \neq f_{\dir{CD}}(D)\end{aligned}$, $f_{\dir{AB}}(C) \neq f_{\dir{AB}}(B)$, and $f_{\dir{AB}}(B) \neq f_{\dir{AB}}(D)$.
\end{proof}

The following theorem provides the criteria for a pair of oriented prime segments to be proper.

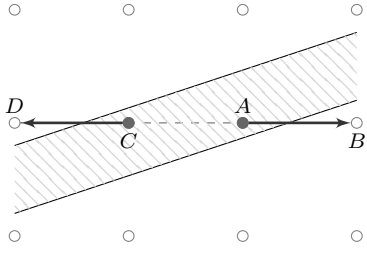
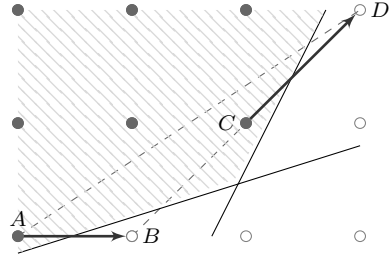
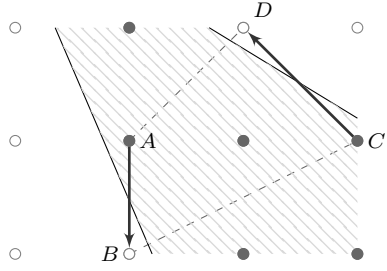
\begin{figure}
\begin{subfigure}[t]{0.48\textwidth}
\captionsetup{aboveskip=20pt}
\centering
\begin{tikzpicture}[scale=1.5]

    	\draw[dashed_line] (0,1) -- (3,1);
    	\draw (0,0.8) -- (3,1.8);
    	\draw (0,0.2) -- (3,1.2);
    	
    \pattern[pattern=north west hatch, hatch distance=2mm, pattern color = light_grey, hatch thickness=.5pt] (0,0.2) -- (0,0.8) -- (3,1.8) -- (3,1.2);

	\foreach \x in {0,...,3}
    	\foreach \y in {0,...,2}
    	{
    		\node[grid_white_point] (\x\y) at (\x,\y) {};
    	}
    	
    	\node[true_point] at (1,1) {};
    	\node[true_point] at (2,1) {};

    	\node[small_text, above] (A) at (2,1) {$A$};	
    	\node[small_text, below] (B) at (3,1) {$B$};	
    	\node[small_text, below] (C) at (1,1) {$C$};	
    	\node[small_text, above] (D) at (0,1) {$D$};	
    	\draw[oriented_segment, name path=AB] (21) -- (31);
    	\draw[oriented_segment, name path=CD] (11) -- (01);

\end{tikzpicture}
\caption{$AC \subset BD$}
\label{fig:AC_in_BD}
\end{subfigure}
\begin{subfigure}[t]{0.48\textwidth}
\captionsetup{aboveskip=20pt}
\centering
\begin{tikzpicture}[scale=1.5]

    \draw[dashed_line, name path=CD] (0,0) -- (3,2);
    \draw[dashed_line, name path=CD] (1,0) -- (2,1);
    	
    \draw (0,-0.15) -- (3,0.8);
    \draw (1.7,0) -- (2.7,2);
    \pattern[pattern=north west hatch, hatch distance=2mm, pattern color = light_grey, hatch thickness=.5pt] (0,-0.15) -- (1.9,0.5) -- (2.7,2) -- (0,2);

	\foreach \x in {0,...,3}
    	\foreach \y in {0,...,2}
    	{
    		\node[grid_point] (\x\y) at (\x,\y) {};
    	}
    	
	\node[grid_white_point] (10) at (1,0) {};    	
    	
    	\foreach \x in {0,...,2}
    	{
    		\node[true_point] at (\x,2) {};
   	}
    	\foreach \x in {0,...,2}
    	{
    		\node[true_point] at (\x,1) {};
    	}
    	
    	\node[true_point] at (0,0) {};

    	\node[small_text, above] (A) at (0,0) {$A$};	
    	\node[small_text, right] (B) at (1,0) {$B$};	
    	\node[small_text, left] (C) at (2,1) {$C$};	
    	\node[small_text, right] (D) at (3,2) {$D$};	
    	\draw[oriented_segment, name path=AB] (00) -- (10);
    	\draw[oriented_segment, name path=CD] (21) -- (32);

\end{tikzpicture}
\caption{$C \in BD$ and $\protect\ortr{A}{B}{D}$ is a counterclockwise triangle}
\label{fig:C_in_BD}
\end{subfigure}
\par\bigskip\bigskip\bigskip
\begin{subfigure}[t]{1\textwidth}
\captionsetup{aboveskip=20pt}
\centering
\begin{tikzpicture}[scale=1.5]

    \draw[dashed_line] (1,0) -- (3,1);
   	\draw[dashed_line] (1,1) -- (2,2);
   	
   	\draw (1.2,0) -- (0.35,2);
   	\draw (3,1.2) -- (1.7, 2);
   	
   	\pattern[pattern=north west hatch, hatch distance=2mm, pattern color = light_grey, hatch thickness=.5pt] (1.2,0) -- (0.35,2) -- (1.7,2) -- (3,1.2) -- (3,0);   

	\foreach \x in {0,...,3}
    	\foreach \y in {0,...,2}
    	{
    		\node[grid_point] (\x\y) at (\x,\y) {};
    	}
    	
    	\node[true_point] at (1,1) {};
    	\node[true_point] at (1,2) {};
   	\node[true_point] at (2,0) {};
   	\node[true_point] at (2,1) {};
   	\node[true_point] at (3,1) {};
   	\node[true_point] at (3,0) {};
   	
   	\node[grid_white_point] (10) at (1,0) {};
   	\node[grid_white_point] (22) at (2,2) {};

   	\node[small_text, right] (A) at (1,1) {$A$};	
   	\node[small_text, left] (B) at (1,0) {$B$};	
    	\node[small_text, right] (C) at (3,1) {$C$};	
   	\node[small_text, above right] (D) at (2,2) {$D$};	
    	\draw[oriented_segment, name path=AB] (11) -- (10);
    	\draw[oriented_segment, name path=CD] (31) -- (22);

\end{tikzpicture}
\caption{$\dir{ABCD}$ is a convex counterclockwise quadrilateral}
\label{fig:ABCD_quadruple}
\end{subfigure}
\caption{The black and white points are the true and false points of $f = f_{\protect\dir{AB}} \protect\land f_{\protect\dir{CD}}$ respectively where $\{\protect\dir{AB},\protect\dir{CD}\}$ is a \superb pair of segments.
The solid lines are separating lines of threshold functions $f_{\protect\dir{AB}}$ and $f_{\protect\dir{CD}}$.
All integer points inside the stripped region and only them are the true points of $f$.}
\end{figure}

\begin{theorem}
\label{th:superb_rectangle}
The pair of prime segments $\dir{AB},\dir{CD}$ is \superb if and only if one of the following holds:
\begin{enumerate}
\item[(1)] $AC \subset BD$;
\item[(2)] $A \in BD$ and $\ortr{C}{D}{B}$ is a counterclockwise triangle or $C \in BD$ and $\ortr{A}{B}{D}$ is a counterclockwise triangle;
\item[(3)] $\dir{ABCD}$ is a counterclockwise quadrilateral.
\end{enumerate}
\end{theorem}
\begin{proof}
Clearly $\Conv(\{A,B,C,D\})$ has at least 2 and at most 4 vertices.
The proof of the theorem is split up into Lemmas \ref{lem:superb_rectangle_1}, \ref{lem:superb_rectangle_2}, and \ref{lem:superb_rectangle_3} according to the number of vertices of $\Conv(\{A,B,C,D\})$.
\end{proof}

The following lemmas treat the cases where $\Conv(\{A,B,C,D\})$ is a segment, triangle, and quadrilateral.

\begin{lemma}
\label{lem:superb_rectangle_1}
A pair of collinear prime segments $\dir{AB},\dir{CD}$ is \superb if and only if  $AC \subset BD$;
\end{lemma}
\begin{proof}
Let $\{\dir{AB},\dir{CD}\}$ be a \superb pair of collinear prime segments (see  \cref{fig:AC_in_BD}). 
Then using \cref{prop:points_on_line} we derive
from $f_{\dir{AB}}(D)=f_{\dir{CD}}(B) = 1$ the inclusion $A,C \in BD$.

Conversely, let $\{\dir{AB},\dir{CD}\}$ be a pair of collinear prime segments with $AC \subset BD$.
The primality of the segments implies that $A \in BC$ and $C \in AD$.
Therefore, by \cref{prop:points_on_line}, we have $f_{\dir{AB}}(C) = f_{\dir{CD}}(A)=f_{\dir{AB}}(D) = f_{\dir{CD}}(B) = 1$, 
and hence the pair $\{\dir{AB},\dir{CD}\}$ is proper, as required.
\end{proof}

\begin{lemma}
\label{lem:superb_rectangle_2}
Let $\{\dir{AB},\dir{CD}\}$ be a pair of prime segments such that \linebreak $\Conv(\{A,B,C,D\})$ is a triangle.
Then the pair is \superb if and only if either $\ortr{C}{D}{B}$ is a counterclockwise triangle with $A \in BD$ or $\ortr{A}{B}{D}$ 
is a counterclockwise triangle with $C \in BD$.
\end{lemma}
\begin{proof}
First assume $\begin{aligned}\{\dir{AB},\dir{CD}\}\end{aligned}$ is a \superb pair of prime segments with \linebreak $\Conv(\{A,B,C,D\})$ being a triangle.
There are four cases to consider:
\begin{enumerate}
\item $D \in \ortr{A}{B}{C}$.
We claim that this case is impossible.
Indeed, if $D$ belongs to the triangle $\ortr{A}{B}{C}$, then $D$ belongs neither to $BC$ nor to $AC$, as otherwise, by \cref{prop:points_on_line}, at least one of $f_{\dir{CD}}(A)$ and $f_{\dir{CD}}(B)$ would be zero, contradicting the assumption that 
$\{\dir{AB},\dir{CD}\}$ is proper.
Therefore, $\ell(CD)$ separates $A$ and $B$, which contradicts $f_{\dir{CD}}(A) = f_{\dir{CD}}(B)$.

\item $B \in \ortr{C}{D}{A}$. 
This case is impossible by similar arguments as in case 1.

\item $C \in \ortr{A}{B}{D}$. We show in this case that $\ortr{A}{B}{D}$ is a counterclockwise triangle and $C \in BD$ (see \cref{fig:C_in_BD}).
The former follows from $f_{\dir{AB}}(D) = 1$. 
To prove the latter, suppose to the contrary that $C \notin BD$.
Then $\ell(BD)$ does not intersect $AC$, and hence, by \cref{prop:theSameOrient}, the orientations of the triangles $\ortr{B}{D}{C}$
and $\ortr{B}{D}{A}$ are the same. Since the orientation of $\ortr{B}{D}{A}$ is the same as that of $\ortr{A}{B}{D}$, we conclude that the orientation 
of $\ortr{B}{C}{D}$ is counterclockwise, and therefore the orientation of $\ortr{C}{D}{B}$ is clockwise, which contradicts $f_{\dir{CD}}(B) = 1$.

\item $A \in \ortr{C}{D}{B}$. In this case arguments similar to the analysis of case 3 show that $\ortr{C}{D}{B}$ is a counterclockwise triangle and $A \in BD$.
\end{enumerate}

Assume now that $\{\dir{AB},\dir{CD}\}$ is a pair of prime segments such that $\ortr{A}{B}{D}$ is a counterclockwise triangle and $C \in BD$.
The case where $\ortr{C}{D}{B}$ is a counterclockwise triangle with  $A \in BD$ is symmetric and we omit the details.
Since $C \in BD$, the orientation of $\ortr{A}{B}{C}$ and $\ortr{C}{D}{A}$ is the same as the orientation of $\ortr{A}{B}{D}$, i.e. counterclockwise.
Consequently, $f_{\dir{AB}}(D) = f_{\dir{AB}}(C) = f_{\dir{CD}}(A) = 1$. 
Furthermore, by \cref{prop:points_on_line}, we have $f_{\dir{CD}}(B) = 1$, and therefore the pair $\{\dir{AB},\dir{CD}\}$ is proper.
\end{proof}

\begin{lemma}
\label{lem:superb_rectangle_3}
Let $\{\dir{AB},\dir{CD}\}$ be a pair of prime segments, such that $A,B,C,$ and $D$ are in convex position.
Then the pair is \superb if and only if $AB$, $BC$, $CD$, $DA$ are edges of  $\Conv(\{A,B,C,D\})$ and the orientation of $\dir{ABCD}$ is counterclockwise.
\end{lemma}
\begin{proof}

First let $\{\dir{AB},\dir{CD}\}$ be a \superb pair of prime segments.
It follows from $f_{\dir{AB}}(C)=f_{\dir{AB}}(D)=f_{\dir{CD}}(A) = f_{\dir{CD}}(B)= 1$ that the triangles $\ortr{A}{B}{C}$, $\ortr{A}{B}{D}$, $\ortr{C}{D}{A}$,
and $\ortr{C}{D}{B}$ are counterclockwise. 
Therefore, by \cref{prop:convex_rectangle}, $AB$, $BC$, $CD$, $DA$ are edges of $\Conv(\{A,B,C,D\})$ and the orientation of $\dir{ABCD}$ is counterclockwise, as required (see \cref{fig:ABCD_quadruple}).

Conversely, let $\dir{ABCD}$ be a counterclockwise quadrilateral.
By definition, the triangles $\ortr{A}{B}{C}$, $\ortr{B}{C}{D}$, $\ortr{C}{D}{A}$, $\ortr{D}{A}{B}$ are counterclockwise.
Therefore 
\[f_{\dir{CD}}(B) = f_{\dir{CD}}(A) = f_{\dir{AB}}(C) = f_{\dir{AB}}(D) = 1,\] 
and hence the pair $\{\dir{AB},\dir{CD}\}$ is proper.
\end{proof}

\cref{th:superb_rectangle} implies a sequence of useful statements about $2$-threshold functions.
The first of them proves that the convex hulls of the sets of true and false points of a function defined by a \superb pair of segments intersect, and hence the function is not threshold.
In other words, a $2$-threshold function defined by a pair of oriented segments is \textit{proper} whenever the pair is \textit{proper}.

\begin{claim}
\label{cor:zeros_ones_intersection}
Let $\{\dir{AB},\dir{CD}\}$ be a \superb pair of segments.
Then $AC \cap BD \neq \emptyset$ and $f = f_{\dir{AB}} \andl f_{\dir{CD}}$ is a proper $2$-threshold function.
\end{claim}
\begin{proof}
By \cref{th:superb_rectangle}, for a \superb pair of segments $\{\dir{AB},\dir{CD}\}$ one of the following statements is true:
\begin{enumerate}
\item[(1)] $AC \subset BD$; in this case $AC \cap BD = AC$.
\item[(2)] $A \in BD$ and $\ortr{C}{D}{B}$ is a counterclockwise triangle or $C \in BD$ and $\ortr{A}{B}{D}$ is counterclockwise triangle; then $AC \cap BD = A$ or $AC \cap BD = C$ respectively.
\item[(3)] $\dir{ABCD}$ is a convex counterclockwise quadrilateral, hence $AC$ and $BD$ are diagonals, and therefore they intersect.
\end{enumerate}
In all cases we have $AC \cap BD \neq \emptyset$, as required.
Since $A,C \in M_1(f)$ and $B,D \in M_0(f)$, we have $\Conv(M_1(f)) \cap \Conv(M_0(f)) \neq \emptyset$, and the function $f$ is not threshold
\end{proof}

\begin{figure}
\centering
\begin{tikzpicture}[scale=1]
	\foreach \x in {0,...,4}
    	\foreach \y in {0,...,3}
    	{
    		\node[grid_point] (\x\y) at (\x,\y) {};
    	}

    	\node[small_text, above] (A) at (1,1) {$B$};	
    	\node[small_text, above] (B) at (3,1) {$D$};	
    	\node[small_text, above] (C) at (2,1) {$A=C$};	
   	\draw[oriented_segment, name path=AB] (21) -- (11);
    	\draw[oriented_segment, name path=CD] (11) -- (31);
    	
    	\node[true_point] at (2,1) {};
    
   	\node[grid_white_point] (11) at (1,1) {};
   	\node[grid_white_point] (31) at (3,1) {};

\end{tikzpicture}
\caption{The function $f$ is true in the unique point $A=C$ and can be defined by a pair of segments~$\{\protect\dir{AB},\protect\dir{CD}\}$.}
\label{fig:singleton}
\end{figure}
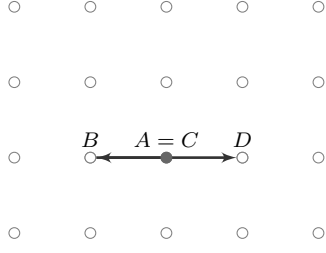

\begin{corollary}
\label{cor:threshold_no_dir_seg}
Every \textup{proper} pair of oriented segments in $\gridMN$ defines a \textup{proper} 2-threshold function on~$\gridMN$.
\end{corollary}

\begin{corollary}
\label{cor:all_ones_on_line}
Let $\{\dir{AB},\dir{CD}\}$ be a \superb pair of collinear segments that define a $2$-threshold function $f$ on $\gridMN$.
Then $M_1(f) = AC \cap \gridMN$ (see \cref{fig:AC_in_BD}).
\end{corollary}

\begin{corollary}
\label{cor:singleton_from_segments}
Let $\{\dir{AB},\dir{AD}\}$ be a \superb pair of segments that define a $2$-threshold function $f$ on $\gridMN$. 
Then $M_1(f) = \{A\}$ (see \cref{fig:singleton}).
\end{corollary}

\begin{corollary}
\label{cor:superb_intersect}
Let $\{\dir{AB},\dir{CD}\}$ be a \superb pair of segments that define a $2$-threshold function $f$ on $\gridMN$.
Then $AB \cap CD \neq \emptyset$ if and only if $M_1(f) = \{A\}$ (see \cref{fig:singleton}).
\end{corollary}

\begin{figure}
\begin{subfigure}[t]{0.45\textwidth}
\centering
\captionsetup{aboveskip=20pt}
\begin{tikzpicture}[scale=1.3]

    \draw[light_grey!50, fill=light_grey!50] (2,0) -- (3,2) -- (1,1);

	\draw (2.17,0) -- (3,2.6);
   	\draw (0,0.6) -- (3, 1.4); 

	\foreach \x in {0,...,3}
    	\foreach \y in {0,...,3}
    	{
    		\node[grid_point] (\x\y) at (\x,\y) {};
    	}
	
	\foreach \x in {0,...,1}
    	\foreach \y in {0,...,3}
    	{
    		\node[grid_white_point] (\x\y) at (\x,\y) {};
    	}

    	\node[small_text, below] at (21) {$A$};
    	\node[small_text, above] at (11) {$B$};
    	\node[small_text, left] at (20) {$C$};
   		\node[small_text, right] at (32) {$D$};
    	\draw[oriented_segment, name path=AB] (21) -- (11);
    	\draw[oriented_segment, name path=CD] (20) -- (32);

\end{tikzpicture}
\caption{
$A$, $B$, $C$, $D$ are in general position. $\protect\mathcal{P}$ is the grey triangle.
}
\label{fig:dir_seg_triangle_in_general}
\end{subfigure}\hfill\hfill
\begin{subfigure}[t]{0.45\textwidth}
\captionsetup{aboveskip=20pt}
\centering
\begin{tikzpicture}[scale=1.3]

	\draw (1.8,0) -- (2.83,3);
   	\draw (0,2.8) -- (3, 0.9);

	\foreach \x in {0,...,3}
    	\foreach \y in {0,...,3}
    	{
    		\node[grid_point] (\x\y) at (\x,\y) {};
    	}

    	\node[small_text, left] at (12) {$A$};
    	\node[small_text, above] at (03) {$B$};
    	\node[small_text, below left] at (21) {$C$};
    	\node[small_text, above] at (33) {$D$};
    	\draw[oriented_segment, name path=AB] (12) -- (03);
    	\draw[oriented_segment, name path=CD] (21) -- (33);

\end{tikzpicture}
\caption{
$A$, $B$, and $C$ are collinear.
}
\label{fig:dir_seg_triangle_3_collinear}
\end{subfigure}
\caption{
The solid lines are the separating lines of the threshold functions defined by the segments $\protect\dir{AB}$ and $\protect\dir{CD}$.}
\label{fig:dir_seg_triangle}
\end{figure}
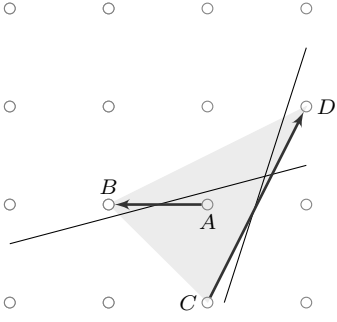
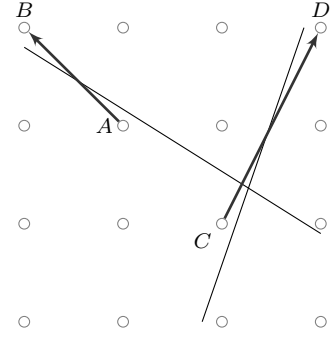

The following claim is related to the property of non-\superb pairs of oriented prime segments.

\begin{claim}
\label{prop:segments_zeros_ones}
Let $\dir{AB}, \dir{CD}$ be distinct prime segments in $\gridMN$ 
such that $f_{\dir{AB}}(C) = 1$, $f_{\dir{AB}}(D) = 0,$ and $f_{\dir{CD}}(A) = 1$.
Then $f_{\dir{CD}}(B) = 1$, the points $B,C,D$ are not collinear, and $A \in \ortr{B}{C}{D}$.
\end{claim}
\begin{proof} 
First we claim that the points $A$, $B$, $C$, $D$ are not collinear. Suppose to the contrary, that they are collinear.
Then, by \cref{prop:points_on_line}, we have $A \in BC$, $B \in AD$, and $C \in AD$, which
imply that either $A=C$ or $A = B$. The latter is not possible as $AB$ is a prime segment. 
Therefore $A=C$ and $B \in CD$. 
Since $CD$ is prime and $C=A \neq B$, we conclude that $B=D$ and $\dir{AB}=\dir{CD}$, 
which contradicts the assumption of the statement.

Assume now that $A,B,C,D$ do not lie on the same line.
From $f_{\dir{AB}}(C) \neq f_{\dir{AB}}(D)$ it follows that $\ell(AB)$ intersects $CD$.
Suppose three of the points $A,B,C,D$ are collinear.
We will consider four cases:
\begin{enumerate}
\item \textit{$A,C,B$ are collinear, i.e. $CD \cap \ell(AB) = C$} (see  \cref{fig:dir_seg_triangle_3_collinear}).
By \cref{prop:points_on_line}, we have $A \in BC$ and hence $A \in \ortr{B}{C}{D}$.
To show $f_{\dir{CD}}(B) = 1$ we observe that the segments $\dir{AB}$ and $\dir{CB}$ are collinear and have the same orientation,
and therefore, by \cref{prop:collinear_segments_and_point}, the triangles $\ortr{A}{B}{D}$ and $\ortr{C}{B}{D}$ have the same orientation.
Since $f_{\dir{AB}}(D) = 0$, the triangle $\ortr{A}{B}{D}$ is clockwise, and hence $\ortr{C}{D}{B}$ is counterclockwise and $f_{\dir{CD}}(B) = 1$.

\item \textit{$A,B,D$ are collinear, i.e. $CD \cap \ell(AB) = D$}.
We will prove that this case is impossible by showing that $\ortr{C}{D}{A}$ is a clockwise triangle, which contradicts $f_{\dir{CD}}(A) = 1$.
By \cref{prop:points_on_line}, we have $B \in AD$, and therefore the segments $\dir{AB}$ and $\dir{AD}$ are collinear and have the same orientation.
Hence, by \cref{prop:collinear_segments_and_point}, the triangles $\ortr{A}{B}{C}$ and $\ortr{A}{D}{C}$ have the same orientation.
Namely, since $f_{\dir{AB}}(C)=1$, we conclude that both triangles are counterclockwise.
Consequently, $\ortr{C}{D}{A}$ is clockwise, as desired.

\item \textit{$A,C,D$ are collinear, i.e. $CD \cap \ell(AB) = A$}.
Since $CD$ is prime and $f_{\dir{CD}}(A) = 1$,
we conclude that $A = C$ and hence the first case takes place.

\item \textit{$C,B,D$ are collinear, i.e. $CD \cap \ell(AB) = B$}.
Since $CD$ is prime and $f_{\dir{AB}}(D) = 0$,
we conclude that $B = D$ and hence the second case takes place.
\end{enumerate}

Assume finally that $A,B,C,D$ are in general position and denote \linebreak $\mathcal{P} = \Conv(\{A,B,C,D\})$ (see \cref{fig:dir_seg_triangle_in_general}).
We consider the oriented triangles $\ortr{C}{D}{A}$, $\ortr{A}{B}{C}$, $\ortr{B}{A}{D}$, and $\ortr{C}{D}{B}$.
It follows from the assumptions of the claim that the first three triangles are counterclockwise.
Therefore, by \cref{prop:clockwise_triangles}, the triangle $\ortr{C}{D}{B}$ is also counterclockwise, 
and hence $f_{\dir{CD}}(B) = 1$.

It remains to show that $A$ belongs to the triangle $\ortr{B}{C}{D}$, i.e.~$\mathcal{P} = \ortr{B}{C}{D}$.
Suppose, to the contrary, $\mathcal{P} \neq \tr{B}{C}{D}$. Then $A$ is a vertex of $\mathcal{P}$ and two of the segments $AC$, $AB$, and $AD$ are edges of $\mathcal{P}$.
We will arrive to a contradiction by showing that neither $AB$ nor $AD$ can be an edge of $\mathcal{P}$.
Indeed, if $AB$ is an edge of $\mathcal{P}$, then $C$ and $D$ are not separated by $\ell(AB)$, which contradicts $f_{\dir{AB}}(C) \neq f_{\dir{AB}}(D)$.
Furthermore, if $AD$ is an edge of $\mathcal{P}$, then $B$ and $C$ are not separated by $\ell(AD)$, and hence the triangles $\ortr{D}{A}{C}$ and 
$\ortr{D}{A}{B}$ have the same orientation. However, the triangle $\ortr{D}{A}{C}$ is counterclockwise as $f_{\dir{CD}}(A) = 1$, and the triangle $\ortr{D}{A}{B}$ is clockwise as $f_{\dir{AB}}(D) = 0$. Contradiction.
\end{proof}

\begin{corollary}
\label{cor:intersection_is_point}
Under the conditions of \cref{prop:segments_zeros_ones} the intersection $\ell(AB) \cap CD$ is a point $X$ such that $A \in XB$.
\end{corollary}

\section{Proper pairs of segments and proper $2$-threshold functions}
\label{sec:proper_functions}

In \cref{cor:zeros_ones_intersection} we showed that proper pairs of segments define proper $2$-threshold functions.
In this section we will prove that the representation by a proper pair of segments is possible for all proper $2$-threshold functions and analyze the number of different proper pairs of segments corresponding to the same function.
In particular, we will show that such a representation is unique for proper $2$-threshold functions with a true point on the boundary of the grid.
We start with the existence of a proper pair of segments for a proper $2$-threshold function.

\begin{theorem}
\label{th:proper_exist}
For any proper $2$-threshold function $f$ on $\gridMN$ there exists a \textup{proper} pair of segments 
in $\gridMN$ that defines $f$.
\end{theorem}
\begin{proof}
Since every proper $2$-threshold function is a conjunction of two \linebreak non-constant threshold functions, it follows from \cref{lem:dir_seg_exist} that there exists a pair of oriented prime segments that defines $f$.
Let $\{\dir{AB}, \dir{CD}\}$ be a pair of oriented prime segments defining $f$ such that $|M_1(f_{\dir{AB}})| + |M_1(f_{\dir{CD}})|$
is minimized. 
We claim that $f_{\dir{CD}}(A) = f_{\dir{AB}}(C) =1$.
For the sake of contradiction, assume without loss of generality that $f_{\dir{CD}}(A) = 0$.
By \cref{th:dir_segment_props}, the point $A$ is essential for $f_{\dir{AB}}$, hence the function $f'$, that differs from $f_{\dir{AB}}$ in the unique point $A$, is threshold.
Since $A \in M_0(f_{\dir{CD}})$ and $M_1(f') = M_1(f_{\dir{AB}}) \setminus \{A\}$, we have 
$$
M_1(f') \cap M_1(f_{\dir{CD}}) = M_1(f_{\dir{AB}}) \cap M_1(f_{\dir{CD}}) = M_1(f),
$$ 
and therefore $f = f' \land f_{\dir{CD}}$.
By assumption $f$ is proper, and hence $f'$ is a non-constant threshold function. 
Consequently, by \cref{lem:dir_seg_exist}, there exists an oriented prime segment $\dir{A'B'}$ that defines $f'$. 
Therefore, the pair $\{\dir{A'B'}, \dir{CD}\}$ defines $f$.
But $|M_1(f')| <  |M_1(f_{\dir{AB}})|$, which contradicts the choice of $\{\dir{AB}, \dir{CD}\}$.

Since $f$ is non-threshold, there exist $X,Y \in M_0(f)$ such that \linebreak $XY \cap \Conv(M_1(f)) \neq \emptyset$.
Indeed, otherwise $\Conv(M_0(f))$ and $\Conv(M_1(f))$ would be disjoint, and therefore separable by a line.
Hence, for any pair of prime segments $\dir{AB}, \dir{CD}$ that defines $f$, neither $f_{\dir{AB}}$ nor $f_{\dir{CD}}$ can be false in both $X,Y$. 
Furthermore, since $X,Y \in M_0(f)$, we conclude that
one of the points is a false point of $f_{\dir{AB}}$ and a true point of $f_{\dir{CD}}$, and
the other point is a true point of $f_{\dir{AB}}$ and a false point of $f_{\dir{CD}}$.

Let $\mathcal{X}$ be the family of \textit{ordered} pairs of segments $\dir{AB},\dir{CD}$ defining $f$
such that $X \in M_0(f_{\dir{AB}}) \cap M_1(f_{\dir{CD}})$ and 
$Y \in M_1(f_{\dir{AB}}) \cap M_0(f_{\dir{CD}})$. 
Denote
$$
M_X = \bigcap\limits_{(\dir{AB},\dir{CD}) \in \mathcal{X}}M_0(f_{\dir{AB}}) \cap M_1(f_{\dir{CD}}).
$$
$$
M_Y = \bigcap\limits_{(\dir{AB},\dir{CD}) \in \mathcal{X}} M_1(f_{\dir{AB}}) \cap M_0(f_{\dir{CD}}).
$$
Notice that each of $M_X$ and $M_Y$ is the intersection of convex sets that have a common element,
and therefore both $M_X$ and $M_Y$ are non-empty and convex.
Moreover, since $M_X, M_Y \subset M_0(f)$,
both $\Conv(M_X)$ and $\Conv(M_Y)$ are disjoint from $\Conv(M_1(f))$.

Let $\ell_X$ be the left inner common tangent to $\Conv(M_1(f))$ and $\Conv(M_X)$.
Let $A^* \in \Conv(M_1(f)) \cap \ell_X$, $B^* \in \Conv(M_X) \cap \ell_X$ be such that $A^*B^*$ is of minimum length.
We claim that $A^*B^*$ is a prime segment.
To prove this, we show first that $\Conv(M_1(f) \cup M_X)$ contains no integer points other than points in $M_1(f) \cup M_X$.
Indeed, let $(\dir{AB}, \dir{CD})$ be a pair of segments from $\mathcal{X}$ and suppose there exists an integer point $Z$
in $\Conv(M_1(f) \cup M_X)$ that belongs neither to $M_1(f)$ nor to $M_X$.
Notice, by definition, $M_X \subset M_1(f_{\dir{CD}})$ and $M_1(f) \subset M_1(f_{\dir{CD}})$, which implies that
$\Conv(M_1(f) \cup M_X) \subseteq \Conv(M_1(f_{\dir{CD}}))$. Consequently, if $f_{\dir{AB}}(Z) = 1$ we have $Z \in M_1(f)$,
and if $f_{\dir{AB}}(Z) = 0$ we have $Z \in M_X$, a contradiction.
Now, any segment with endpoints in $M_1(f) \cup M_X$ belongs to $\Conv(M_1(f) \cup M_X)$, hence if there is an integer point $Z$ in the interior of $A^*B^*$ then $Z \in M_1(f) \cup M_X$, which contradicts the minimality of $A^*B^*$.
Similarly, considering the left inner common tangent $\ell_Y$ to $\Conv(M_1(f))$ and $\Conv(M_Y)$, the two points
$C^* \in M_1(f) \cap \ell_Y$, $D^* \in M_Y \cap \ell_Y$ at minimum distance define a prime segment $C^*D^*$.
\cref{fig:Mx_My} illustrates $M_X, M_Y, A^*, B^*, C^*$, and $D^*$.

\begin{figure}
\centering
\begin{subfigure}[t]{0.45\textwidth}
\captionsetup{aboveskip=20pt}
\centering
\begin{tikzpicture}[scale=0.9]
	\draw[light_grey, fill=light_grey!50] (1,4) -- (3,2) -- (4,4) -- (1,4);
	\draw[gray] (0,0) -- (1,0) -- (2,1) -- (2,2) -- (0,4) -- (0,0);
	\draw[gray] (6,0) -- (6,4) -- (5,4) -- (4,3) -- (4,2) -- (6,0);
	
    	\draw[dashed_line, name path=lx] (0.8,-0.2) -- (5.2,4.2);
    	\draw[dashed_line, name path=ly] (0.9,4.2) -- (3.1,-0.2);
    
    	\path [name intersections={of=lx and ly, by=O}];
    	\pattern[pattern=north east hatch, hatch distance=2mm, pattern color = light_grey, hatch thickness=.5pt] (1,4) -- (O) -- (5,4);
    	
	\draw[oriented_segment, name path=AB] (3,2) -- (4,3);
    	\draw[oriented_segment, name path=CD] (1,4) -- (2,2);

    	\node[true_point] (C) at (1,4) {};
    	\node[true_point] (A) at (3,2) {};
    	\node[false_point] (Y) at (0.2,3) {};
    	\node[false_point] (X) at (5.8,3) {};
    	\node[grid_white_point] (D) at (2,2) {};
    	\node[grid_white_point] (B) at (4,3) {};
   	\node[true_point] (Z) at (4.3,3.7) {};

    	\node[small_text, below right] at (A) {$A^*$};	
    	\node[small_text, below right] at (B) {$B^*$};	
    	\node[small_text, left] at (C) {$C^*$};	
    	\node[small_text, left] at (D) {$D^*$};	
    	\node[small_text, above] at (Y) {$Y$};	
    	\node[small_text, above] at (X) {$X$};	
    	\node[small_text, right] at (Z) {$Z$};	
    	\node[small_text, above right] (Mx) at (4.7,2) {$M_X$};	
    	\node[small_text, above right] (My) at (0.5,1) {$M_Y$};
    	\node[small_text, below left] (M1) at (2.8,3.8) {$M_1(f)$};	
    	\node[small_text, above right] (l_x) at (5.2,4.2) {$\ell_X$};	
    	\node[small_text, above right] (l_y) at (3.1,-0.2) {$\ell_Y$};	

\end{tikzpicture}
\caption{All integer points of the stripped region are exactly the true points of $f^*$. $Z$ is chosen outside of $\protect\Conv(M_1(f))$ and such that $f^*(Z) = 1$.}
\label{fig:Mx_My_a}
\end{subfigure}\hfill\hfill
\begin{subfigure}[t]{0.45\textwidth}
\captionsetup{aboveskip=20pt}
\centering
\begin{tikzpicture}[scale=0.9]

	\draw[light_grey, fill=light_grey!50] (1,4) -- (3,2) -- (4,4) -- (1,4);
	\draw[gray] (0,0) -- (1,0) -- (2,1) -- (2,2) -- (0,4) -- (0,0);
	\draw[gray] (6,0) -- (6,4) -- (5,4) -- (4,3) -- (4,2) -- (6,0);
	\pattern[pattern=north west hatch, hatch distance=2mm, pattern color = light_grey, hatch thickness=.5pt] (2,2) -- (3,2) -- (1,4);
	\pattern[pattern=north east hatch, hatch distance=2mm, pattern color = light_grey, hatch thickness=.5pt] (1,4) -- (3,2) -- (4,3) -- (4.3,3.7) -- (4,4) -- (1,4);
	
    	\draw[dashed_line, name path=lx] (0.8,-0.2) -- (5.2,4.2);
    	\draw[dashed_line, name path=ly] (0.9,4.2) -- (3.1,-0.2);

	\draw[oriented_segment, name path=AB] (3,2) -- (4,3);
   	\draw[oriented_segment, name path=CD] (1,4) -- (2,2);
    
    	\path [name intersections={of=lx and ly, by=O}];

	\draw[black!60] (C) -- (A);
	\draw[black!60] (D) -- (Z);
    	
    	\node[true_point] (C) at (1,4) {};
    	\node[true_point] (A) at (3,2) {};
    	\node[false_point] (Y) at (0.2,3) {};
    	\node[false_point] (X) at (5.8,3) {};
     	\node[grid_white_point] (D) at (2,2) {};
    	\node[grid_white_point] (B) at (4,3) {};   	
    	\node[true_point] (Z) at (4.3,3.7) {};
    
    	\node[small_text, below right] at (A) {$A^*$};	
    	\node[small_text, below right] at (B) {$B^*$};	
    	\node[small_text, left] at (C) {$C^*$};	
    	\node[small_text, left] at (D) {$D^*$};	
    	\node[small_text, above] at (Y) {$Y$};	
    	\node[small_text, above] at (X) {$X$};	
    	\node[small_text, right] at (Z) {$Z$};	
    	\node[small_text, above right] (Mx) at (4.7,2) {$M_X$};	
    	\node[small_text, above right] (My) at (0.5,1) {$M_Y$};
    	\node[small_text, below left] (P) at (3.2,3.8) {$\mathcal{P}$};	
    	\node[small_text, above right] (l_x) at (5.2,4.2) {$\ell_X$};	
    	\node[small_text, above right] (l_y) at (3.1,-0.2) {$\ell_Y$};	

\end{tikzpicture}
\caption{The pair $\protect\dir{A^*B^*}$, $\protect\dir{C^*D^*}$ is proper. The stripped region is $\protect\mathcal{P}$. $S_1$ and $S_2$ have the different pattern orientation. The segment $D^*Z$ intersects $A^*C^*$.}
\label{fig:Mx_My_b}
\end{subfigure}
\caption{The white polygons are $\protect\Conv(M_X)$ and $\protect\Conv(M_Y)$. The grey polygon is $\protect\Conv(M_1(f))$.}
\label{fig:Mx_My}
\end{figure}
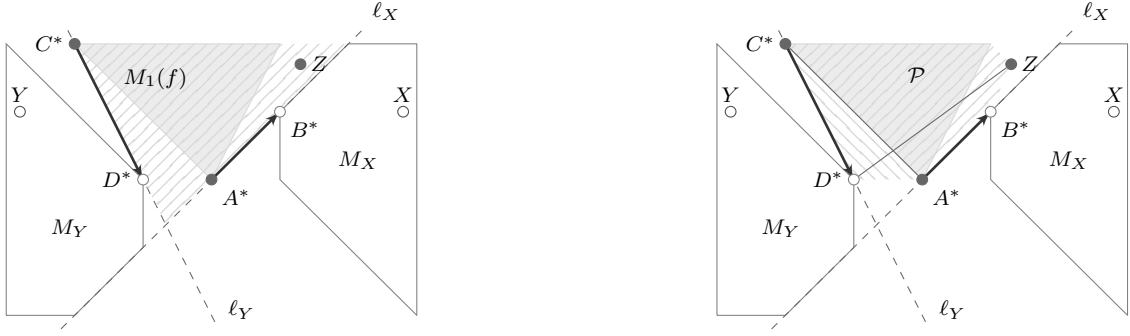

Let now $f^* = f_{\dir{A^*B^*}} \land f_{\dir{C^*D^*}}$ be the 2-threshold function defined by \linebreak $\{\dir{A^*B^*}, \dir{C^*D^*}\}$.
In the rest of the proof we will show that $f = f^*$ and the pair $\{\dir{A^*B^*}, \dir{C^*D^*}\}$ is proper. 
To establish the former we will prove that $M_1(f) = M_1(f^*)$.

First we show that $M_1(f) \subseteq M_1(f^*)$.
Indeed, by definition, $\ell(\dir{A^*B^*}) = \ell_X$ is a left tangent from $B^*$ to $\Conv(M_1(f))$, and therefore
$M_1(f) \subseteq M_1(f_{\dir{A^*B^*}})$.
Similarly, we have $M_1(f) \subseteq M_1(f_{\dir{C^*D^*}})$, and therefore 
$M_1(f) \subseteq M_1(f_{\dir{A^*B^*}}) \cap M_1(f_{\dir{C^*D^*}}) = M_1(f^*)$.

Now, let us show that $M_1(f^*) \subseteq M_1(f)$.
Assume, to the contrary, $M_1(f^*) \setminus M_1(f) \neq \emptyset$ and let $Z$ be
a point in $M_1(f^*) \setminus M_1(f)$.
In particular, we have $Z \notin M_X \cup M_Y$.
We observe that $f(Z) = 0$ and $Z \notin M_Y$ imply that there exists a pair $(\dir{AB}, \dir{CD}) \in \mathcal{X}$ such that $Z \in M_0(f_{\dir{AB}})$, and therefore $M_X \cup \{Z\} \subseteq M_0(f_{\dir{AB}})$ and 
\begin{equation}\label{eq:MxZ}
	\Conv(M_X \cup \{Z\})\cap \Conv(M_1(f)) = \emptyset.
\end{equation}
Similarly, it can be shown that 
\begin{equation}\label{eq:MyZ}
	\Conv(M_Y \cup \{Z\}) \cap \Conv(M_1(f)) = \emptyset.
\end{equation}

We will consider two cases depending on whether $\{\dir{A^*B^*},\dir{C^*D^*}\}$ is a \superb pair or not.
We start with the case of \superb pair, in which case we have $f_{\dir{A^*B^*}}(D^*) = f_{\dir{C^*D^*}}(B^*) = 1$ (see \cref{fig:Mx_My_a}).
First we claim that $A^* \neq C^*$. Indeed, otherwise, by \cref{cor:singleton_from_segments}, we would have $M_1(f^*) = \{A^*\}$, and therefore since $M_1(f) \subseteq M_1(f^*)$ and $M_1(f^*) \setminus M_1(f) \neq \emptyset$, we would conclude that $f$ is the constant-zero function, contradicting the assumption that $f$ is a proper 2-threshold function.
Let us now denote $\mathcal{P} = \Conv(M_1(f^*) \cup \{B^*,D^*\})$.
From $M_1(f) \cup \{D^*\} \subseteq M_1(f_{\dir{A^*B^*}})$ and $A^*,B^* \in \ell_X$ it follows that $\ell_X$ is a tangent to $\mathcal{P}$ where $A^*$ is a tangent point.
Analysis similar to the above implies that $\ell_Y$ is a tangent to $\mathcal{P}$ and $C^*$ is a tangent point.
Consequently, all points of $\mathcal{P} \setminus A^*C^*$ are separated by the segment $A^*C^*$ into two parts, which we denote as $S_1$ and $S_2$ (see \cref{fig:Mx_My_b}).
By \cref{cor:zeros_ones_intersection}, the segments $A^*C^*$ and $B^*D^*$ intersect, and hence $B^*$ and $D^*$ are in different parts, say $B^* \in S_1$ and $D^*\in S_2$.
We now claim that $Z$ belongs to one of the parts $S_1$ and $S_2$.
To see this, we first observe that $Z \in M_1(f^*) \subseteq \mathcal{P}$. Furthermore, since $Z$ belongs to
$M_0(f)$, it does not belong to $A^*C^*$, and hence the claim.
Now, assume without loss of generality $Z \in S_1$, and therefore $D^*Z$ intersects $A^*C^*$.
Since $D^* \in M_Y$ and $A^*C^* \subseteq \Conv(M_1(f))$, we conclude that $\Conv(M_Y \cup \{Z\}) \cap \Conv(M_1(f)) \neq \emptyset$, which contradicts (\ref{eq:MyZ}).

\begin{figure}
\centering
\begin{subfigure}[t]{0.45\textwidth}
\captionsetup{aboveskip=20pt}
\centering
\begin{tikzpicture}[scale=0.9]

	\draw[light_grey, fill=light_grey!50] (1,4) -- (3,2) -- (4,4) -- (1,4);
	\draw[gray] (0,0) -- (1,0) -- (2.75,0.5) -- (1.7,2) -- (0,4) -- (0,0);
	\draw[gray] (6,0) -- (6,4) -- (5,4) -- (4,3) -- (4,2) -- (6,0);
	\pattern[pattern=north west hatch, hatch distance=2mm, pattern color = light_grey, hatch thickness=.5pt] (O) -- (3,2) -- (1,4);
	\pattern[pattern=north east hatch, hatch distance=2mm, pattern color = light_grey, hatch thickness=.5pt] (1,4) -- (3,2) -- (4,3) -- (4.3,3.7) -- (4,4) -- (1,4);
	
    	\draw[dashed_line, name path=lx] (0.8,-0.2) -- (5.2,4.2);
    	\draw[dashed_line, name path=ly] (0.9,4.2) -- (3.1,-0.2);

	\draw[oriented_segment, name path=AB] (A) -- (B);
    	\draw[oriented_segment, name path=CD] (C) -- (2.75,0.5);
    
	\draw[name path=YZ,black!60]  (Y) -- (Z);    
    
    	\path [name intersections={of=lx and ly, by=O}];
   	\path [name intersections={of=ly and YZ, by=V}];

	\draw[black!60] (C) -- (A);
	\draw[black!60] (O) -- (Z);

    	\node[true_point] (C) at (1,4) {};
    	\node[true_point] (A) at (3,2) {};
    	\node[false_point] (Y) at (0.2,3) {};
    	\node[false_point] (X) at (5.8,3) {};
    	\node[grid_white_point] (D) at (2.75,0.5) {};
    	\node[grid_white_point] (B) at (4,3) {};
    	\node[true_point] (Z) at (4.3,3.7) {};
    	\node[intersection_point] at (O) {};
    	\node[intersection_point] at (V) {};
    
    	\node[small_text, below right] at (A) {$A^*$};	
    	\node[small_text, below right] at (B) {$B^*$};	
    	\node[small_text, left] at (C) {$C^*$};	
    	\node[small_text, left] at (D) {$D^*$};	
    	\node[small_text, above] at (Y) {$Y$};	
    	\node[small_text, above] at (X) {$X$};	
    	\node[small_text, left] at (O) {$O$};	
   	\node[small_text, above left] at (V) {$V$};	
    	\node[small_text, right] at (Z) {$Z$};	
    	\node[small_text, above right] (Mx) at (4.7,2) {$M_X$};	
    	\node[small_text, above right] (My) at (0.5,1) {$M_Y$};
    	\node[small_text, below left] (P) at (3.2,3.8) {$\mathcal{P}$};		
    	\node[small_text, above right] (l_x) at (5.2,4.2) {$\ell_X$};	
    	\node[small_text, above right] (l_y) at (3.1,-0.2) {$\ell_Y$};	

\end{tikzpicture}
\caption{$A^*\neq C^*$, $OZ \subseteq \Conv(M_Y \protect\cup \{Z\})$.}
\label{fig:Mx_My_non_superb_AnotC}
\end{subfigure}\hfill\hfill
\begin{subfigure}[t]{0.45\textwidth}
\captionsetup{aboveskip=20pt}
\centering
\begin{tikzpicture}[scale=0.9]

	\draw[light_grey, fill=light_grey!50] (3.34,4) -- (3,2) -- (4,4);
	\draw[gray] (0,0) -- (1,0) -- (2.75,0.5) -- (1.7,2) -- (0,4) -- (0,0);
	\draw[gray] (6,0) -- (6,4) -- (5,4) -- (4,3) -- (4,2) -- (6,0);
	\pattern[pattern=north east hatch, hatch distance=2mm, pattern color = light_grey, hatch thickness=.5pt] (3.34,4) -- (3,2) -- (4,3) -- (4.3,3.7) -- (4,4) -- (3.34,4);
	
    	\draw[dashed_line, name path=lx] (0.8,-0.2) -- (5.2,4.2);
    	\path (A) -- (2.75,0.5) coordinate[pos=-1.6](dd) coordinate[pos=1.4](ff);
	\draw[dashed_line, name path=lineCD] (dd) -- (A);
	\draw[dashed_line] (2.75,0.5) -- (ff);

	\draw[oriented_segment, name path=AB] (A) -- (B);
    	\draw[oriented_segment, name path=CD] (A) -- (2.75,0.5);
    
	\draw[name path=YZ,black!60]  (Y) -- (Z);    
    
    	\path [name intersections={of=lineCD and YZ, by=V}];

	\draw[black!60] (A) -- (Z);

    	\node[true_point] (A) at (3,2) {};
    	\node[false_point] (Y) at (0.2,3) {};
    	\node[false_point] (X) at (5.8,3) {};
    	\node[grid_white_point] (D) at (2.75,0.5) {};
    	\node[grid_white_point] (B) at (4,3) {};
    	\node[true_point] (Z) at (4.3,3.7) {};
    	\node[intersection_point] at (V) {};
    
    	\node[small_text, below right] at (A) {$A^*=C^*$};	
    	\node[small_text, below right] at (B) {$B^*$};	
    	\node[small_text, left] at (D) {$D^*$};	
    	\node[small_text, above] at (Y) {$Y$};	
    	\node[small_text, above] at (X) {$X$};		
    	\node[small_text, above left] at (V) {$V$};	
    	\node[small_text, right] at (Z) {$Z$};	
    	\node[small_text, above right] (Mx) at (4.7,2) {$M_X$};	
    	\node[small_text, above right] (My) at (0.5,1) {$M_Y$};	
    	\node[small_text, above right] (l_x) at (5.2,4.2) {$\ell_X$};	
    	\node[small_text, above right] at (ff) {$\ell_Y$};	

\end{tikzpicture}
\caption{$A^*=C^*$, $A^* \in \Conv(M_Y \cup \{Z\})$.}
\label{fig:Mx_My_non_superb_AisC}
\end{subfigure}
\caption{The pair $\{\protect\dir{A^*B^*}$, $\protect\dir{C^*D^*}\}$ is not \superb and $f_{\protect\dir{A^*B^*}}(D^*) = 0$. 
The grey region is $\protect\Conv(M_1(f))$. 
The stripped region is $\protect\mathcal{P}$. 
$S_1$ and $S_2$ have the different pattern orientation.}
\label{fig:Mx_My_non_superb}
\end{figure}
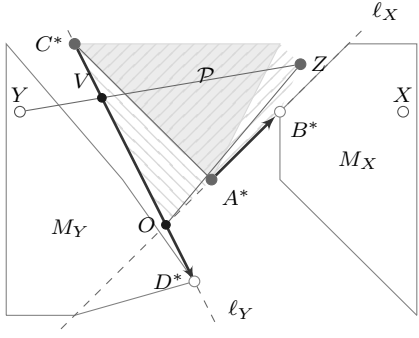
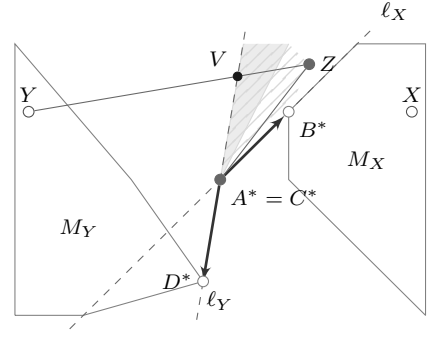

Suppose now that the pair $\{\dir{A^*B^*},\dir{C^*D^*}\}$ is not proper, which implies that
$f_{\dir{A^*B^*}}(D^*) = 0$ or $f_{\dir{C^*D^*}}(B^*) = 0$. 
There is no loss of generality in assuming $f_{\dir{A^*B^*}}(D^*) = 0$ (see  \cref{fig:Mx_My_non_superb_AnotC}).
Then \cref{prop:segments_zeros_ones} yields $f_{\dir{C^*D^*}}(B^*) = 1$.
Let $A^* \neq C^*$, the case $A^* = C^*$ will be considered separately.
From $f_{\dir{A^*B^*}}(C^*)\neq f_{\dir{A^*B^*}}(D^*)$ it follows that $\ell_X$ intersects $C^*D^*$.
We denote $O = \ell_X \cap C^*D^*$ and consider $\mathcal{P} = \Conv(M_1(f^*) \cup \{B^*,O\})$.
As in the previous case it can be verified that $\ell_X, \ell_Y$ are tangents to $\mathcal{P}$, and therefore $A^*$ and $C^*$ are tangent points. 
Thus the points of $\mathcal{P} \setminus A^*C^*$ are separated by $A^*C^*$ into two parts,  which we denote as $S_1$ and $S_2$.
We next prove that $O$ and $B^*$ are in different parts.
For this purpose, we consider the triangle $\ortr{B^*}{C^*}{D^*}$, and \cref{prop:segments_zeros_ones} implies $A^* \in \ortr{B^*}{C^*}{D^*}$.
It is easily seen that $OB^* = \ortr{B^*}{C^*}{D^*} \cap \ell(A^*B^*)$, hence $A^* \in OB^*$, and therefore $O$ and $B^*$ belong to the different parts, say $B^* \in S_1$ and $O \in S_2$.
Clearly, $Z \in \mathcal{P} \setminus A^*C^*$, and therefore either $Z \in S_1$ or $Z \in S_2$.
The latter would contradict (\ref{eq:MxZ}), so we assume the former holds, which in turn implies $OZ \cap A^*C^* \neq \emptyset$.
To obtain a contradiction with (\ref{eq:MyZ}) we will show $OZ \subseteq \Conv(M_Y \cup \{Z\})$.
To this end we first observe that $\ell_Y$ intersects $YZ$ because $f_{\dir{C^*D^*}}(Y) \neq f_{\dir{C^*D^*}}(Z)$. 
Let $V$ be the intersection point of $YZ$ and $\ell_Y$.
Now from $f_{\dir{A^*B^*}}(Y) = f_{\dir{A^*B^*}}(Z) = 1$ it follows that $V \in \Conv(M_1(f_{\dir{A^*B^*}}))$.
Since $D^* \in M_0(f_{\dir{A^*B^*}})$, we conclude that $\ell_X$ intersects $D^*V$  and $O \in D^*V$.
But $D^*V \subseteq \ortr{Y}{D^*}{Z} \subseteq \Conv(M_Y \cup \{Z\})$, and therefore $O \in \Conv(M_Y \cup \{Z\})$ and $OZ \subseteq \Conv(M_Y \cup \{Z\})$, leading to a contradiction.
Suppose now that $A^*=C^*$ (see \cref{fig:Mx_My_non_superb_AisC}). By replacing $O$ with $A^*$, and using arguments similar to the above one can show that $A^* \in \ortr{Y}{D^*}{Z}$ and $A^*Z \subseteq \Conv(M_Y \cup \{Z\})$, which contradicts (\ref{eq:MyZ}).
The contradictions in all the cases imply that $M_1(f^*) \setminus M_1(f) = \emptyset$, and hence $f = f^*$.

We have shown that $\{\dir{A^*B^*},\dir{C^*D^*}\}$ defines $f$. 
It remains to prove that \linebreak $\{\dir{A^*B^*},\dir{C^*D^*}\}$ is a \superb pair of segments.
Since $B^* \in M_X$ and $B^* \in M_0(f_{\dir{A^*B^*}})$, the definition of $M_X$ implies that $f_{\dir{C^*D^*}}(B^*) = 1$.
Similarly, from $D^*\in M_Y$ and $D^* \in M_0(f_{\dir{C^*D^*}})$ we conclude $f_{\dir{A^*B^*}}(D^*) = 1$.
Finally, the equality $f_{\dir{A^*B^*}}(C^*) = f_{\dir{C^*D^*}}(A^*) = 1$ follows from $A^*,C^* \in M_1(f)$.
Hence $\{\dir{A^*B^*},\dir{C^*D^*}\}$ is a \superb pair of segments that defines $f$, as claimed.
\end{proof}

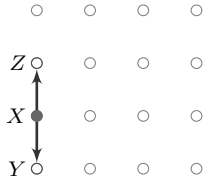
\begin{figure}
\centering
\begin{tikzpicture}[scale=0.7]
	\foreach \x in {0,...,3}
    	\foreach \y in {0,...,3}
    	{
    		\node[grid_point] (\x\y) at (\x,\y) {};
    	}
	\node[true_point] (X) at (01)  {};
	\node[false_point] (Y) at (00)  {};
	\node[false_point] (Z) at (02) {};
		
    	\draw[oriented_segment, name path=XY] (X) -- (Y);
    	\draw[oriented_segment, name path=XZ] (X) -- (Z);
	
    	\node[small_text, left] at (X) {$X$};
    	\node[small_text, left] at (Y) {$Y$};
    	\node[small_text, left] at (Z) {$Z$};
\end{tikzpicture}
\caption{$\{\protect\dir{XY}, \protect\dir{XZ}\}$ defines a 2-threshold function $f$ such that $M_1(f) = \{X\}$.}
\label{fig:singleton_on_boundary}
\end{figure}

When representing proper $2$-threshold functions via proper pairs of segments, it is natural to ask whether such a representation is unique or not? 
The examples in \cref{fig:two_pairs_dir_segments_gen} present proper $2$-threshold functions with at least two distinct proper pairs of segments representing them.
Moreover, the following statement shows that the number of distinct proper pairs of segments defining the same singleton-function can be as large as $\Theta(mn)$:
\begin{figure}
\begin{subfigure}[t]{0.48\textwidth}
	\captionsetup{aboveskip=20pt}
\centering
\begin{tikzpicture}[scale=1.1]
	\foreach \x in {0,...,4}
    	\foreach \y in {0,...,3}
    	{
    		\node[grid_point] (\x\y) at (\x,\y) {};
    	}

    	\node[small_text, above left] at (22) {$A$};
    	\node[small_text, right] at (23) {$B$};
    	\draw[oriented_segment, name path=AB] (22) -- (23);
    	\path (22) -- (23) coordinate[pos=-1.4](dd) coordinate[pos=2](ff);
	\draw[dashed_line, name path=lineAB] (dd) -- (22);
	\draw[dashed_line] (23) -- (ff);
	\node[true_point] (22) at (22) {};

    	\node[small_text, right] at (21) {$D$};
    	\draw[oriented_segment, name path=A'B'] (22) -- (21);
    	\path (22) -- (21) coordinate[pos=-1.4](dd) coordinate[pos=2](ff);
	\draw[dashed_line, name path=lineA'B'] (dd) -- (22);
	\draw[dashed_line] (21) -- (ff);

    	\node[small_text, above] at (43) {$B'$};
    	\draw[oriented_segment, name path=CD] (22) -- (43);
    	\path (22) -- (43) coordinate[pos=-0.7](dd) coordinate[pos=1.5](ff);
	\draw[dashed_line, name path=lineCD] (dd) -- (22);
	\draw[dashed_line] (43) -- (ff);

    	\node[small_text, left] at (01) {$D'$};
    	\draw[oriented_segment, name path=C'D'] (22) -- (01);
    	\path (22) -- (01) coordinate[pos=-0.6](dd) coordinate[pos=1.2](ff);
	\draw[dashed_line, name path=lineC'D'] (dd) -- (22);
	\draw[dashed_line] (01) -- (ff);
	
	\node[grid_white_point] (21) at (21) {};
	\node[grid_white_point] (23) at (23) {};

\end{tikzpicture}
\caption{\footnotesize{$M_1(f) = \{A\}$, $A=A'=C=C'=X=Y=Z=U$.}}
\label{fig:two_pairs_dir_segments_singleton}
\end{subfigure}\hfill\hfill
\begin{subfigure}[t]{0.48\textwidth}
\captionsetup{aboveskip=20pt}
\centering
\begin{tikzpicture}[scale=1.1]
	\foreach \x in {0,...,4}
    	\foreach \y in {0,...,3}
    	{
    		\node[grid_point] (\x\y) at (\x,\y) {};
    	}

    	\node[small_text, below right] at (11) {$C=C'$};
    	\node[small_text, above] at (43) {$D'$};
    	\draw[oriented_segment, name path=AB] (11) -- (43);
    	\path (11) -- (43) coordinate[pos=-0.4](dd) coordinate[pos=0.5](ff);
	\draw[dashed_line, name path=lineAB] (dd) -- (11);
	\draw[dashed_line] (43) -- (ff);
	\node[true_point] (11) at (11) {};

    	\node[small_text, above left] at (22) {$A=A'$};
    	\node[small_text, right] at (33) {$B$};
    	\draw[oriented_segment, name path=A'B'] (22) -- (33);
    	\path (22) -- (33) coordinate[pos=-1.4](dd) coordinate[pos=2](ff);
	\draw[dashed_line, name path=lineA'B'] (dd) -- (22);
	\draw[dashed_line] (33) -- (ff);
	\node[true_point] (22) at (22) {};

    	\node[small_text, above] at (01) {$B'$};
    	\draw[oriented_segment, name path=CD] (22) -- (01);
    	\path (22) -- (01) coordinate[pos=-0.7](dd) coordinate[pos=1.5](ff);
	\draw[dashed_line, name path=lineCD] (dd) -- (22);
	\draw[dashed_line] (01) -- (ff);

    	\node[small_text, left] at (00) {$D$};
    	\draw[oriented_segment, name path=C'D'] (11) -- (00);
    	\path (11) -- (00) coordinate[pos=-0.6](dd) coordinate[pos=1.2](ff);
	\draw[dashed_line, name path=lineC'D'] (dd) -- (11);
	\draw[dashed_line] (00) -- (ff);

\end{tikzpicture}
\caption{\footnotesize{$M_1(f) = \{A,C\}$, $A=A'=X=Y,C=C'=Z=U$.}}
\label{fig:two_pairs_dir_segments_line}
\end{subfigure}
\caption{Examples of $2$-threshold functions with two distinct \superb pairs of segments.}
\label{fig:two_pairs_dir_segments_gen}
\end{figure}
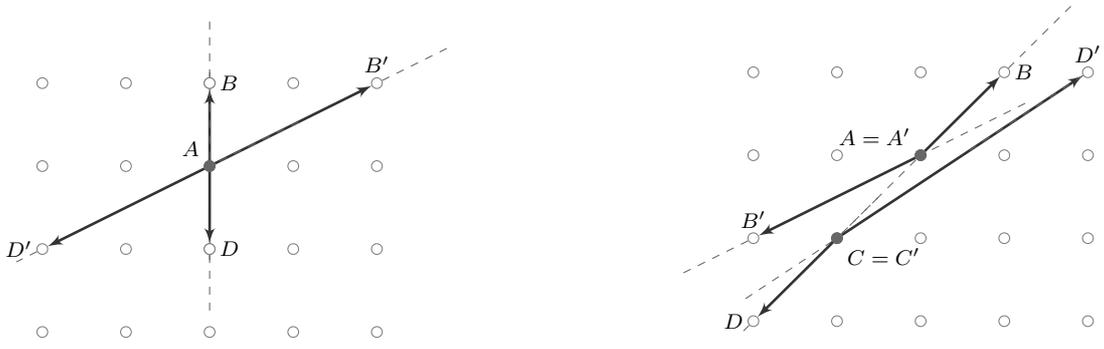

\begin{claim}
Let $f$ be a $\{ 0,1 \}$-valued function on $\gridMN$ with a unique true point $A=(a_1,a_2)$ such that $a_1 \in \{1,\dots,m-2\}$ and $a_2 \in \{1,\dots,n-2\}$. 
Then $f$ is a $2$-threshold function, the number of \superb pairs of segments defining $f$ is at most
$$
\frac{3}{\pi^2}mn + O(m\log n),
$$
and this upper bound is achievable by some functions.
\end{claim}
\begin{proof}
Without loss of generality we assume 
\begin{equation}
\label{eq:a1a2}
a_1 \leq \frac{m-1}{2}, a_2 \leq \frac{n-1}{2}.
\end{equation} 
Let $\dir{AB}$ and $\dir{AD}$ be distinct prime segments.
By \cref{th:superb_rectangle}, the pair $\{\dir{AB},\dir{AD}\}$ is \superb if and only if both segments belong to the same line.
Hence, if $\{\dir{AB},\dir{AD}\}$ is proper, then $d(\dir{AB}) = d(\dir{AD})$, and therefore all the considered  pairs of segments belong to a subgrid of size $(2a_1+1) \times (2a_2+1)$.
Next, we notice that for any given \superb pair $\{\dir{AB},\dir{AD}\}$ the points $B$ and $D$ are symmetric to 
each other with respect to $A$.
Therefore it is enough to estimate the number of choices for $B$. 
Let $B = (b_1,b_2), D = (d_1,d_2)$.
The only \superb pair with $b_1 = d_1$ is the pair where $\{B, D\} = \{(a_1, a_2 + 1),(a_1, a_2 - 1)\}$, so we can exclude this case and assume $b_1 \neq d_1$.
By symmetry, we may also assume $b_1 < d_1$.

Putting all together and using a standard number-theoretical formula
\begin{flalign*}
\sum_{p = 1}^{m}\sum_{\substack{q = 1 \\ q \perp p}}^{n}1 = \frac{6}{\pi^2} mn + O(m \log n)
\end{flalign*}

we derive the number of possible choices for $B$ (see \cref{fig:singleton_inside_grid}):
$$
\sum_{b_1 = 0}^{a_1-1}\sum_{\substack{b_2 = 0 \\ (b_1-a_1) \perp (b_2-a_2)}}^{2a_2 + 1}1 = \sum_{p = 1}^{a_1}\sum_{\substack{q = -a_2 \\ p \perp q}}^{a_2 + 1}1 =  \frac{12}{\pi^2}a_1a_2 + O(a_1\log a_2).
$$

The target estimation follows from the latter by replacing $a_1, a_2$ with their upper bound~(\ref{eq:a1a2}).
\end{proof}

\begin{figure}
	\captionsetup{aboveskip=20pt}
\centering
\begin{tikzpicture}[scale=0.7]
	
	\foreach \x in {0,...,12}
    	\foreach \y in {0,...,7}
    	{
    		\node[grid_point] (\x\y) at (\x,\y) {};
    	}

	\node[true_point] at (43) {};
	\draw[dashed_line] (80) -- (86);
    \draw[dashed_line] (06) -- (86);	
	\draw[dashed_line] (00) -- (06);
	\draw[dashed_line] (00) -- (80);

	\draw[oriented_segment] (43) -- (33);
	\draw[oriented_segment] (43) -- (32);
	\draw[oriented_segment] (43) -- (31);
	\draw[oriented_segment] (43) -- (30);
	\draw[oriented_segment] (43) -- (34);
	\draw[oriented_segment] (43) -- (35);
	\draw[oriented_segment] (43) -- (36);
	\draw[oriented_segment] (43) -- (20);
	\draw[oriented_segment] (43) -- (22);
	\draw[oriented_segment] (43) -- (24);
	\draw[oriented_segment] (43) -- (26);
	\draw[oriented_segment] (43) -- (11);
	\draw[oriented_segment] (43) -- (12);
	\draw[oriented_segment] (43) -- (14);
	\draw[oriented_segment] (43) -- (15);
	\draw[oriented_segment] (43) -- (00);
	\draw[oriented_segment] (43) -- (02);
	\draw[oriented_segment] (43) -- (04);
	\draw[oriented_segment] (43) -- (06);
    	
	\node[small_text, right] at (43) {$A$};

\end{tikzpicture}
\caption{\footnotesize{For $A = (4,3)$, all the \superb pairs of segments belong to the subgrid with the dashed boundary and $A$ in the center. The possible choices of $B$ are drawn on the left half of the subgrid.}}
\label{fig:singleton_inside_grid}
\end{figure}
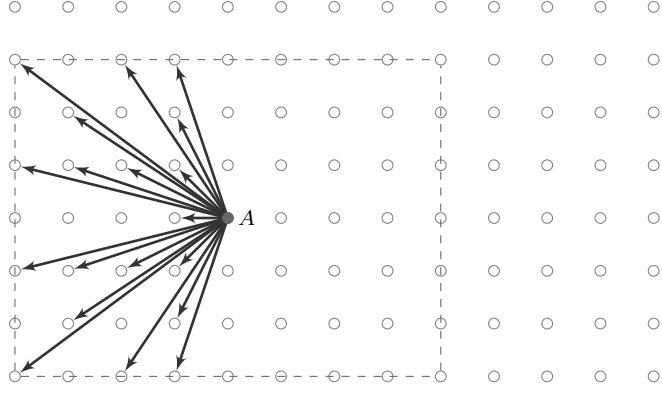

Although the representation via proper pair of segments is not unique for some $2$-threshold functions, it is unique for the functions that contain at least one true point on the boundary of the grid.
In the following lemma we prove this fact for the special case of singleton-functions, and then proceed with the general case.

\begin{lemma}
\label{prop:one_proper_for_point}
Let $f$ be a $\{ 0,1 \}$-valued function on $\gridMN$ with a unique true point $X=(x_1,x_2)$ such that 
either $x_1 \in \{0,m-1\}$ or $x_2 \in \{0,n-1\}$, but not both. 
Then $f$ is a proper $2$-threshold function with a unique \superb pair of segments defining $f$.
\end{lemma}
\begin{proof}
Due to symmetry it is enough to consider the case $x_1=0$ and $x_2 \in \{1,\dots,n-2\}$.
We will show that $\{\dir{XY}, \dir{XZ}\}$, where $Y = (0, x_2-1)$, $Z = (0,x_2+1)$,
is the desired pair (see \cref{fig:singleton_on_boundary}).
In \cite{Zamaraeva2016} it was proved that any $\{0,1\}$-function containing one true point is $k$-threshold for any $k \geq 2$, hence $f$ is a $2$-threshold function.
From \cref{th:superb_rectangle} and \cref{cor:singleton_from_segments} it follows that the pair $\{\dir{XY}, \dir{XZ}\}$ is \superb and defines $f$, and therefore $f$ is non-threshold.
Now, let us prove that there is no other \superb pair of segments that defines $f$.

Let $\{\dir{XY'},\dir{XZ'}\}$ be a \superb pair segments that defines $f$. 
We will show that $\{Y',Z'\} = \{Y,Z\}$.
First, $f(Z) = 0$ implies that $f_{\dir{XY'}}(Z)=0$ or $f_{\dir{XZ'}}(Z) = 0$.
Without loss of generality we assume $f_{\dir{XZ'}}(Z)=0$.
Since both $\dir{XZ}$ and $\dir{XZ'}$ are prime, we conclude that either $Z' = Z$ or $\ortr{X}{Z'}{Z}$ is a clockwise triangle.
For the sake of contradiction, let us assume the latter holds.
By definition of a clockwise triangle,
$$\begin{vmatrix}
0 & x_2 & 1 \\ 
z_1 & z_2 & 1 \\ 
0 & x_2+1 & 1
\notag
\end{vmatrix} = z_1 < 0,$$
where $Z' = (z_1,z_2)$.
But this contradicts $z_1\geq 0$, hence $Z'=Z$. Now let us show that $Y'=Y$.
Indeed, as $\{\dir{XY'}, \dir{XZ}\}$ is a \superb pair, by definition,
$Y' \in M_1(f_{\dir{XZ}}) = \{(0,0),(0,1),\dots,(0,x_2)\}$, and
therefore, since $\dir{XY'}$ is prime and $X = (0,x_2)$, we conclude that $Y' = (0,x_2-1) = Y$.
\end{proof}

\begin{theorem}
\label{prop:two_proper_only_inside}
For any proper 2-threshold function $f$ on $\gridMN$ that contains a true point on the boundary of $\gridMN$ 
there exists a \textup{unique} \superb pair of segments in $\gridMN$ that defines $f$.
\end{theorem}
\begin{proof}
By \cref{th:proper_exist}, there exists at least one \superb pair of segments that defines $f$.
Suppose, for the sake of contradiction, that there are two different \superb pairs of segments defining $f$, which
we denote as $\{\dir{AB}, \dir{CD}\}$ and $\{\dir{A'B'}, \dir{C'D'}\}$ respectively.

First we will prove that 
\begin{equation}
\label{eq:empty_intersection}
\{\dir{AB}, \dir{CD}\} \cap \{\dir{A'B'}, \dir{C'D'}\} = \emptyset.
\end{equation}
Suppose, to the contrary, that $\dir{AB}=\dir{A'B'}$, then $\dir{CD} \neq \dir{C'D'}$.
Since $f_{\dir{AB}}(D)=f_{\dir{AB}}(D') = 1$ and $f(D) = f(D') = 0$, we have $f_{\dir{C'D'}}(D) = f_{\dir{CD}}(D') = 0$.
Furthermore, $f(C)=f(C')=1$ implies $f_{\dir{C'D'}}(C)=f_{\dir{CD}}(C')=1$.
On the other hand, by \cref{prop:segments_zeros_ones}, the equations $f_{\dir{C'D'}}(C)=1, f_{\dir{CD}}(C')=1, f_{\dir{CD}}(D') = 0$ imply $f_{\dir{C'D'}}(D) =1$, a contradiction.

Now we will look more closely at the functions $f_{\dir{AB}}, f_{\dir{CD}}, f_{\dir{A'B'}}$, and $f_{\dir{C'D'}}$.
Since $f(B) = 0$, we have either $f_{\dir{A'B'}}(B) = 0$ or $f_{\dir{C'D'}}(B)= 0$.
Without loss of generality we assume $f_{\dir{A'B'}}(B) = 0$.
From $f_{\dir{AB}}(A') = 1, f_{\dir{A'B'}}(A) = 1, f_{\dir{A'B'}}(B) = 0$, and \cref{prop:segments_zeros_ones} it follows that the points $A,B,B'$ are not collinear and 
$f_{\dir{AB}}(B') = 1$.
The latter together with the fact that $f(B') = 0$ imply $f_{\dir{CD}}(B') = 0$.
By \cref{cor:intersection_is_point}, the line $\ell(A'B')$ intersects $AB$ in a unique point, which we denote by $X$, and $A' \in XB'$.

Analysis similar to above shows that $f_{\dir{CD}}(B') = 0$ implies $f_{\dir{C'D'}}(D) = 0$ and 
that the line $\ell(CD)$ intersects $A'B'$ in a unique point, which we denote by $Y$, and $C \in YD$.
In turn, the equation $f_{\dir{C'D'}}(D) = 0$ implies $f_{\dir{AB}}(D') = 0$ and 
the intersection of $\ell(C'D')$ and $CD$ in a unique point denoted by $Z$, and $C' \in ZD'$.
Finally, the equation $f_{\dir{AB}}(D') = 0$ implies that $\ell(AB)$ intersects $C'D'$ in a unique point denoted by $U$, and $A \in UB$.

\begin{figure}
\centering
\begin{tikzpicture}[scale=0.95]

	\draw[light_grey, fill=light_grey] (3,2) -- (5,3) -- (6,4) -- (6,7) -- (4,7) -- (3,6) -- (3,2);

    	
    	\node[small_text, below] at (3,2) {$A$};
    	\node[small_text, above] at (7,3) {$B$};
    	\draw[oriented_segment, name path=AB] (3,2) -- (7,3);
    	\path (3,2) -- (7,3) coordinate[pos=-0.7](dd) coordinate[pos=1.4](ff);
	\draw[dashed_line, name path=lineAB] (dd) -- (3,2);
	\draw[dashed_line] (7,3) -- (ff);
	\node[true_point] (32) at (3,2) {};
	\node[grid_white_point] (73) at (7,3) {};

    	\node[small_text, right] at (6,4) {$A'$};
    	\node[small_text, right] at (7,8) {$B'$};
    	\draw[oriented_segment, name path=A'B'] (6,4) -- (7,8);
    	\path (6,4) -- (7,8) coordinate[pos=-0.7](dd) coordinate[pos=1.25](ff);
	\draw[dashed_line, name path=lineA'B'] (dd) -- (6,4);
	\draw[dashed_line] (7,8) -- (ff);
	\node[grid_white_point] (64) at (6,4) {};
	\node[grid_white_point] (78) at (7,8) {};

    	\node[small_text, above] at (6,7) {$C$};
    	\node[small_text, above] at (2,8) {$D$};
    	\draw[oriented_segment, name path=CD] (6,7) -- (2,8);
    	\path (6,7) -- (2,8) coordinate[pos=-0.7](dd) coordinate[pos=1.5](ff);
	\draw[dashed_line, name path=lineCD] (dd) -- (6,7);
	\draw[dashed_line] (2,8) -- (ff);
	\node[true_point] (67) at (6,7) {};
	\node[grid_white_point] (28) at (2,8) {};

    	\node[small_text, left] at (3,6) {$C'$};
    	\node[small_text, left] at (2,1) {$D'$};
    	\draw[oriented_segment, name path=C'D'] (3,6) -- (2,1);
    	\path (3,6) -- (2,1) coordinate[pos=-0.6](dd) coordinate[pos=1.1](ff);
	\draw[dashed_line, name path=lineC'D'] (dd) -- (3,6);
	\draw[dashed_line] (2,1) -- (ff);
	\node[true_point] (36) at (3,6) {};
	\node[grid_white_point] (21) at (2,1) {};

	\path [name intersections={of=lineA'B' and AB, by=X}];
	\node[intersection_point] (X) at (X) {};
	\node[small_text, below right] at (X) {$X$};	
	\path [name intersections={of=lineCD and A'B', by=Y}];
	\node[intersection_point] (Y) at (Y) {};
	\node[small_text, below right] at (Y) {$Y$};
	\path [name intersections={of=lineC'D' and CD, by=Z}];
	\node[intersection_point] (Z) at (Z) {};
	\node[small_text, above right] at (Z) {$Z$};
	\path [name intersections={of=lineAB and C'D', by=U}];
	\node[intersection_point] (U) at (U) {};
	\node[small_text, below right] at (U) {$U$};
	
\end{tikzpicture}
\caption{The grey region is $\Conv(M_1(f))$, which is included in $\Conv(\{X,Y,Z,U\})$.}
\label{fig:two_pairs_dir_segments}
\end{figure}
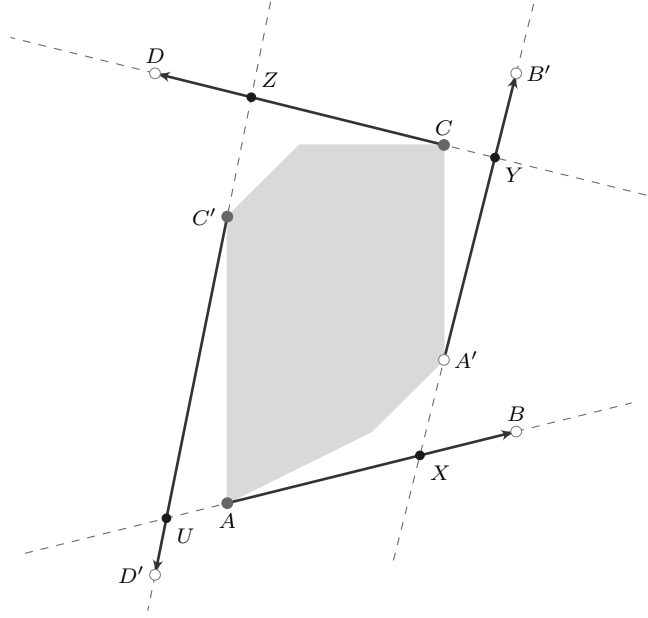

In the rest of the proof we will show that $M_1(f) \subseteq \Conv(\{X,Y,Z,U\})$ and that $X$, $Y$, $Z$, $U$ are interior points of $\Conv(\gridMN)$, which will lead to a contradiction (see \cref{fig:two_pairs_dir_segments}).
We will consider four different cases.

\medskip
\noindent
\textbf{Case 1.} \textit{The points $X,Y,Z,U$ are pairwise distinct.}
First we will show that \linebreak $\Conv(\{X,Y,Z,U\})$ is a counterclockwise quadrilateral with the edges $XY$, $YZ$, $ZU$, and $UX$ (see  \cref{fig:two_pairs_dir_segments}).
Applied to $f_{\dir{AB}},f_{\dir{C'D'}}$, \cref{prop:segments_zeros_ones} yields $A \in \ortr{B}{C'}{D'}$, and hence $A \in UB$.
The latter together with $X \in AB$ imply that $\dir{AB}$ and $\dir{UX}$ have the same orientation.
By similar arguments, $\dir{A'B'}$ and $\dir{XY}$, $\dir{CD}$ and $\dir{YZ}$, and $\dir{C'D'}$ and $\dir{ZU}$ have the same orientation respectively.
Now we observe that the assumption $Y \neq Z$ implies $Z \not\in \ell(A'B')$. 
Therefore, since $f_{\dir{A'B'}}(C) = f_{\dir{A'B'}}(D) = 1$ and
$Z \in CD$, the triangle $\ortr{A'}{B'}{Z}$ is counterclockwise.
Hence, by \cref{prop:collinear_segments_and_point}, the triangle $\ortr{X}{Y}{Z}$ is counterclockwise.
By similar arguments, the triangles $\ortr{Y}{Z}{U}$, $\ortr{Z}{U}{X}$, $\ortr{U}{X}{Y}$ are counterclockwise.
Consequently, by \cref{prop:convex_rectangle}, $\Conv(\{X,Y,Z,U\})$ is a quadrilateral $XYZU$ with edges $XY$, $YZ$, $ZU$, $UX$.

Next, the inclusion $\Conv(M_1(f)) \subseteq XYZU$ follows from the fact that $XYZU$ is a polygon circumscribed about 
$\Conv(M_1(f))$.
Indeed, each of the lines 
$\ell(A'B') = \ell(XY)$, $\ell(CD) = \ell(YZ)$, $\ell(C'D') = \ell(ZU)$, and $\ell(AB) = \ell(UX)$
is a tangent to $\Conv(M_1(f))$, and $A' \in XY \cap \Conv(M_1(f))$, $C \in YZ \cap \Conv(M_1(f))$, $C' \in ZU \cap \Conv(M_1(f))$, $A \in UX \cap \Conv(M_1(f))$.

It remains to prove that all the points $X,Y,Z$, and $U$ are interior points of $\Conv(\gridMN)$, i.e.
$X,Y,Z,U \notin B(\gridMN)$, where 
\[B(\gridMN) = \{0,m-1\}\times[0,n-1] \cup [0,m-1]\times\{0,n-1\}.
\]
We will prove that $X \notin B(\gridMN)$, for the other three points the arguments are similar.
Suppose, to the contrary, that $X \in B(\gridMN)$.
Since $X \in AB$ and $A \in UB$, we have $X \in UB$.
We claim that $X$ is an interior point of $UB$. Indeed, $X \neq U$ by the assumption.
Furthermore, the equality $X = B$ would imply $A' \in BB'$, which is not possible as $f_{\dir{A'B'}}$ is a threshold function
and $f_{\dir{A'B'}}(B) = 0, f_{\dir{A'B'}}(A') = 1, f_{\dir{A'B'}}(B') = 0$.
Now, since both $U$ and $B$ belong to $\Conv(\gridMN)$, and $X$ is an interior point of $UB$ and 
a boundary point of $\Conv(\gridMN)$, we conclude that $\ell(UB)=\ell(AB)$ is a tangent to $\Conv(\gridMN)$.
We will arrive to a contradiction by showing that $\ell(AB)$ separates $D$ and $D'$.
First, we observe that $D' \notin \ell(AB)$, as otherwise we would have $U=D'$ and $A \in D'B$, which is not possible as $f_{\dir{AB}}$ is threshold and $f_{\dir{AB}}(B)=0, f_{\dir{AB}}(A)=1, f_{\dir{AB}}(D')=0$.
Consequently, $\ortr{A}{B}{D'}$ is a clockwise triangle. On the other hand, the triangle $\ortr{A}{B}{D}$ is counterclockwise as the pair $\{\dir{AB},\dir{CD}\}$ is proper.
Therefore, $\ell(AB)$ separates $D$ and $D'$. This contradiction proves that $X$ does not belong to $B(\gridMN)$.

\medskip
\noindent
\textbf{Case 2.} \textit{$X=Z$ or $Y=U$.}
Suppose $X = Z$. Then from $X \in AB$ and $Z \in CD$ it follows that $AB$ and $CD$ intersect.
However, $\{\dir{AB},\dir{CD}\}$ is a \superb pair of segments, and, by \cref{cor:superb_intersect}, we have $M_1(f) = \{A\}$ (see  \cref{fig:two_pairs_dir_segments_singleton}).
Since $f$ is a proper 2-threshold function, $A$ is not a vertex of $\Conv(\gridMN)$, and therefore
\cref{prop:one_proper_for_point} implies $A \in  \{1,\dots,m-2\}\times\{1,\dots,n-2\}$, as required.
The case $Y = U$ is symmetric and we omit the details.

\medskip
\noindent
\textbf{Case 3.} \textit{$|\{X,Y,Z,U\}| = 3$, $X \neq Z$, and $Y \neq U$.}
Let $X = Y$, using the same arguments as in Case 1 it can be shown that $\ortr{X}{Z}{U}$ is a triangle circumscribed about $\Conv(M_1(f))$, and that none of $X,Z$,
and $U$ lies on the boundary of $\gridMN$.
The cases $X = U$, $Y = Z$, and $Z = U$ are symmetric and we omit the details.

\medskip
\noindent
\textbf{Case 4.} \textit{$|\{X,Y,Z,U\}| = 2$ and $X \neq Z, Y \neq U$.}
Then either $X=Y$ and $U = Z$ or $X=U$ and $Y=Z$.
The two cases are symmetric and therefore we consider only one of them, namely, $X=Y$, $U = Z$.
First we will show that $\Conv(M_1(f)) = AC$.
Indeed, from $X \in AB$, $Y \in A'B'$, and $A' \in \ortr{A}{B}{B'}$ it follows that $X = Y = A'$, and hence $A = A'$ as $AB$ is prime.
Moreover, $Y \in \ell(CD)$ together with $Y = A$ imply that $A,C,D$ are collinear points, and hence $\Conv(\{A,B,C,D\})$ has at most three vertices.
Then, by \cref{th:superb_rectangle}, either $A \in BD$ or $C \in BD$ or both.
All cases lead to the conclusion that $A,B,C,D$ are collinear, and, by \cref{cor:all_ones_on_line}, we have $\Conv(M_1(f)) = AC$ (see  \cref{fig:two_pairs_dir_segments_line}).

Now, it remains to show that $A,C \notin B(\gridMN)$.
Conversely, suppose $A \in B(\gridMN)$ or $C \in B(\gridMN)$. 
Without loss of generality we assume the former, which
in turn implies that $\ell(AB)$ is a tangent to $\Conv(\gridMN)$ as $A$ is an interior point of $BD$ and $B,D \in \gridMN$.
We will arrive to a contradiction by showing that $\ell(AB)$ separates $B'$ and $D'$.
For this we observe that neither $\dir{A'B'}$ nor $\dir{C'D'}$ belongs to $\ell(AB)$. Indeed, as by \cref{th:superb_rectangle}
$AC \subset BD$, the inclusion $A'B' \subset \ell(AB)$ would imply that $A'B'$ coincides either with $AB$ or with $CD$,
and the inclusion $C'D' \subset \ell(AB)$ would imply that $C'D'$ coincides either with $AB$ or with $CD$.
In each of the cases we would have a contradiction with (\ref{eq:empty_intersection}).
This observation together with the fact that $A',C' \in M_1(f) \subseteq AC \subset \ell(AB)$ imply that neither $B'$ nor $D'$
belongs to $\ell(AB)$.
Consequently, as $f_{\dir{AB}}$ takes different values in $B'$ and $D'$ we conclude that $\ell(AB)$ separates $B'$ and $D'$, as required.
\end{proof}

\section{Conclusion}

In this paper we characterized $2$-threshold functions via pairs of oriented prime segments.
We introduced the notion of proper pairs of segments and showed that the endpoints of the segments in a proper pair of segments are essential for the function they define.
In this way a $2$-threshold function $f$ can be defined by a partially ordered set of $4$ essential points.
This representation is more space efficient than the representations by a minimal specifying set of $f$ or by the convex polygon $\Conv(M_1(f))$ (see \cite{Zunic1995, Zamaraeva2017}).

We also established the existence of a defining proper pair of segments for every proper $2$-threshold function and the uniqueness of such a pair for the functions with a true point on the boundary of the grid.
In fact, in the subsequent work \cite{part2} we prove that almost all $2$-threshold functions have a true point on the boundary of the grid and use the uniqueness of their representation to derive the first asymptotic formula for the number of $2$-threshold functions.


It is natural to wonder whether the approach we used to characterize $2$-threshold functions can be generalized to higher order threshold functions, say to $3$-threshold functions.
One difference between $2$-threshold and $3$-threshold functions that might be an obstacle towards such a generalization is an observation that for $3$-threshold functions the requirement to have a true point on the boundary of the grid is more restrictive than for $2$-threshold functions.
This is an issue for future research to explore.

\bibliographystyle{unsrt}

\end{document}